\documentclass[11pt,a4paper,final]{article}

\usepackage{amsmath}    
\usepackage{amsfonts}   
\usepackage{amssymb}    
\usepackage{amsthm}     
\usepackage{mathrsfs}   
\usepackage{graphicx}   
\usepackage{color}
\usepackage{bm}
\usepackage[square,sort,comma,numbers]{natbib}
\usepackage[left=3cm,right=3cm,top=3cm,bottom=3cm]{geometry}

\numberwithin{equation}{section}

\newcommand{\ie}{{\it i.e.}}
\newcommand{\cf}{{\it c.f.}}
\newcommand{\eg}{{\it e.g.}}

\newcommand{\etc}{{\it etc.}}
\newcommand{\iid}{{\it i.i.d.}}

\newcommand{\Hajek}{H\'{a}jek}
\newcommand{\Holder}{H\"{o}lder }

\newtheorem{theorem}{Theorem}[section]

\newtheorem{lemma}{Lemma}[section]

\newtheorem{proposition}{Proposition}[section]

\newtheorem{corollary}{Corollary}[section]

\theoremstyle{remark}

\newtheorem{example}{\bf Example}[section]

\renewenvironment{proof}{\noindent{\it Proof.}\;}{\qed}

\newcommand{\score}{\dot{\ell}}

\newcommand{\scrA}{{\mathscr A}}
\newcommand{\scrF}{{\mathscr F}}
\newcommand{\scrP}{{\mathscr P}}
\newcommand{\scrX}{{\mathscr X}}

\newcommand{\NN}{{\mathbb N}}
\newcommand{\RR}{{\mathbb R}}
\newcommand{\DD}{{\mathbb D}}

\newcommand{\cF}{\mathcal F}
\newcommand{\cG}{\mathcal G}
\newcommand{\cH}{\mathcal H}
\newcommand{\cS}{\mathcal S}
\newcommand{\cW}{\mathcal W}
\newcommand{\cL}{\mathcal L}

\newcommand{\bX}{{\bf X}}
\newcommand{\bZ}{{\bf Z}}

\newcommand{\bepsilon}{\mbox{\boldmath{$\epsilon$}}}
\newcommand{\bmu}{\mbox{\boldmath{$\mu$}}}

\newcommand{\bc}{\begin{center}}
\newcommand{\ec}{\end{center}}
\newcommand{\be}{\begin{equation}}
\newcommand{\ee}{\end{equation}}
\newcommand{\ba}{\begin{array}}
\newcommand{\ea}{\end{array}}
\newcommand{\bean}{\begin{eqnarray*}}
\newcommand{\eean}{\end{eqnarray*}}
\newcommand{\bea}{\begin{eqnarray}}
\newcommand{\eea}{\end{eqnarray}}
\newcommand{\ben}{\begin{enumerate}}
\newcommand{\een}{\end{enumerate}}
\newcommand{\bed}{\begin{itemize}}
\newcommand{\eed}{\end{itemize}}

\begin{document}

\thispagestyle{empty}

\title{\vspace*{-9mm}
  The semi-parametric Bernstein-von Mises theorem\\
  for regression models with symmetric errors}
\author{Minwoo Chae$^{1}$, Yongdai Kim$^{2}$ and Bas Kleijn$^{3}$\\[2mm]
  {\small\it {}$^{1}$\, Department of Mathematics, University of Texas %
    at Austin}\\
  {\small\it {}$^{2}$\, Department of Statistics, Seoul National University}\\
  {\small\it {}$^{3}$\, Korteweg-de~Vries Institute for Mathematics,
    University of Amsterdam}
  }

\date{\today}
\maketitle

\begin{abstract}
In a smooth semi-parametric model, the marginal posterior distribution for
a finite dimensional parameter of interest is expected to be asymptotically
equivalent to the sampling distribution of any efficient point-estimator.
The assertion leads to asymptotic equivalence of credible and confidence
sets for the parameter of interest and is known as the semi-parametric
Bernstein-von Mises theorem. In recent years, it has received much
attention and has been applied in many examples. We consider models in
which errors with symmetric densities play a role; more specifically,
it is shown that the marginal posterior distributions of regression
coefficients in the linear regression and linear mixed effect models
satisfy the semi-parametric Bernstein-von Mises assertion. As a consequence,
Bayes estimators in these models achieve frequentist inferential
optimality, as expressed \eg\ through \Hajek's convolution and asymptotic
minimax theorems. Conditions for the prior on the space of error
densities are relatively mild and well-known constructions like the
Dirichlet process mixture of normal densities and random series priors
constitute valid choices. Particularly, the result provides an efficient
estimate of regression coefficients in the linear mixed effect model,
for which no other efficient point-estimator was known previously.
\end{abstract}


\section{Introduction}

In this paper, we give an asymptotic, Bayesian analysis of models with
errors that are distributed symmetrically.
The observations $\bX = (X_{1}, \ldots, X_{n})^T \in \RR^n$ are modeled by,
\begin{equation}
\label{eq:general_model}
	\bX = \bmu + \bepsilon,
\end{equation}
where $\bmu = (\mu_{1}, \ldots, \mu_{n})^T$ and $\bepsilon = (\epsilon_{1},
\ldots, \epsilon_{n})^T$. Here the mean vector $\bmu$ is non-random and
parametrized by a finite dimensional parameter $\theta$, and the
distribution of the error vector $\bepsilon$ is symmetric in the sense that
$\bepsilon$ has the same distribution as $-\bepsilon$. Since the error
has a symmetric but otherwise unknown distribution, the model is
semi-parametric. Examples of models of the form (\ref{eq:general_model})
are the symmetric location model (where $\mu_i=\theta\in \RR$,), and the
linear regression model (where $\mu_i=\theta^T Z_i$ for given covariates
$Z_i\in \RR^p$). Moreover, the form (\ref{eq:general_model}) includes
models with dependent errors, like linear mixed effect models.

The main goal of this paper is to prove the semi-parametric
Bernstein-von Mises (BvM) assertion for models of the form
(\ref{eq:general_model}) with symmetric error distributions.
Roughly speaking we show that the marginal posterior distribution of
the parameter of interest $\theta$ is asymptotically normal, centered on an
efficient estimator with variance equal to the inverse Fisher information
matrix. As a result, statistical inference based on the posterior
distribution satisfies frequentist criteria of optimality. 

Various sets of sufficient conditions for the semi-parametric BvM theorem
based on the full LAN (local asymptotic normality) expansion 
(\ie\ the LAN expansion with respect to both the finite and infinite
dimensional parameters \cite{mcneney2000application}) have been developed
in \cite{shen2002asymptotic, bickel2012semiparametric, castillo2015bernstein}.
The full LAN expansion, however, is conceptually inaccessible and technically
difficult to verify. Because the models we consider are adaptive
\cite{bickel1982adaptive}, we can consider a simpler type of LAN
expansion that involves only the parameter of interest, albeit that the
expansion must be valid under data distributions that differ slightly
from the one on which the expansion is centred. We call this property
{\it misspecified LAN} and prove that it holds for the models of the
form (\ref{eq:general_model}) and that, together with other regularity
conditions, it implies the semi-parametric BvM assertion. 

While the BvM theorem for parametric Bayesian models is well established 
(\eg\ \cite{le1990asymptotic, kleijn2012bernstein}), the semi-parametric BvM
theorem is still being studied very actively: initial examples
\cite{cox1993analysis, freedman1999wald} of simple semi-parametric
problems with simple choices for the prior demonstrated failures of
marginals posteriors to display BvM-type asymptotic behaviour.
Subsequently, positive semi-parametric BvM results have been established
in these and various other examples, including models in survival
analysis (\cite{kim2004bernstein, kim2006bernstein}), multivariate
normal regression models with growing numbers of parameters
(\cite{bontemps2011bernstein, johnstone2010high, ghosal1999asymptotic})
and discrete probability measures (\cite{boucheron2009bernstein}). More
delicate notions like finite sample properties and second-order
asymptotics are considered in \cite{panov2015finite, spokoiny2013bernstein,
yang2015semiparametric}. 

Regarding models of the form
(\ref{eq:general_model}), there is a sizable amount of literature on
efficient point-estimation in the symmetric location problem
(\cite{beran1978efficient, stone1975adaptive, sacks1975asymptotically})
and linear regression models (\cite{bickel1982adaptive}). By contrast,
to date {\it no efficient point-estimator} for the regression coefficients
in the linear mixed effect model has been found; the semi-parametric BvM
theorem proved below, however, implies that the Bayes estimator is
efficient! To the authors' best knowledge, this paper provides the first
efficient semi-parametric estimator in the linear mixed effect model. A
numerical study given in section~\ref{sec:simulation} supports the
view that the Bayes estimator is superior to previous methods of estimation.

This paper is organized as follows:
section~\ref{sec:general_BvM} proves the semi-parametric BvM assertion
for all smooth adaptive models (\cf\ the misspecified LAN expansion). In
sections~\ref{sec:regression} and~\ref{sec:lm} we study the linear
regression model and linear mixed effect model, respectively. For each,
we consider two common choices for the nuisance prior, a Dirichlet process
mixture and a series prior, and we show that both lead to validity of
the BvM assertion. Results of numerical studies are presented in
section~\ref{sec:simulation}. 

\subsubsection*{Notation and conventions}

For two real values $a$ and $b$, $a\wedge b$ and $a\vee b$ are
the minimum and maximum of $a$ and $b$, respectively, and 
$a_n \lesssim b_n$ signifies that $a_n$ is smaller than $b_n$ up
to a constant multiple independent of $n$.
Lebesgue measures are denoted by $\mu$; $|\cdot|$ represents the Euclidean
norm on $\RR^d$. 
The capitals
$P_\eta$, $P_{\theta,\eta}$ \etc\ denote the probability measures
associated with
densities that we write in lower case, $p_\eta$, $p_{\theta,\eta}$ \etc\
(where it is always clear from the context which dominating measure $\mu$
is involved). The corresponding log densities are indicated with $\ell_\eta$,
$\ell_{\theta,\eta}$ \etc\
Hellinger and total-variational metrics are defined as
$h^2(p_1,p_2) = \int \big(\sqrt{p_1} - \sqrt{p_2}\big)^2 d\mu$
and $d_V(p_1, p_2) = \int |p_1 - p_2| d\mu$,
respectively. 
The expectation of a random variable $X$ under a probability measure
$P$ is denoted by $P X$. The notation $P_0$ always represents the true
probability which generates the observation and $X^o = X - P_0 X$ is
the centered version of a random variable $X$. 
The indicator function for a set $A$ is denoted $1_A$.
For a class of measurable functions
$\cF$, the quantities $N(\epsilon, \cF, d)$ and $N_{[\,]}(\epsilon, \cF, d)$
represent the $\epsilon$-covering and -bracketing numbers
\cite{van1996weak} with respect to a (semi)metric $d$.

\section{Misspecified LAN and the semi-parametric BvM theorem}
\label{sec:general_BvM}

In this section, we prove the semi-parametric BvM theorem for
smooth adaptive models, \ie\ those that satisfy the misspecified
LAN expansion defined below.

\subsection{Misspecified local asymptotic normality}

Consider a sequence of statistical models
$\scrP^{(n)} = \{P_{\theta,\eta}^{(n)}: \theta \in \Theta, \eta \in \cH\}$ 
on measurable spaces $(\scrX^{(n)},\scrA^{(n)})$, parametrized by a finite
dimensional parameter $\theta$ of interest and an infinite dimensional
nuisance parameter $\eta$. Assume that $\Theta$ is a subset of
$\RR^p$, $\cH$ is a metric space equipped with the associated Borel
$\sigma$-algebra and $P_{\theta,\eta}^{(n)}$ has density $x \mapsto
p_{\theta,\eta}^{(n)}(x)$ with respect to some $\sigma$-finite measures
$\mu^{(n)}$ dominating $\scrP^{(n)}$. 

Let $X^{(n)}$ be a $\scrX^{(n)}$-valued
random element following $P_0^{(n)}$ and assume that
$P_0^{(n)} = P_{\theta_0,\eta_0}^{(n)}$ for some $\theta_0\in\Theta$ and
$\eta_0 \in \cH$. We say that a sequence of statistical models
$\scrP^{(n)}$ satisfies the
{\it misspecified LAN expansion} if there exists a sequence of vector-valued
(componentwise) $L_2(P_0^{(n)})$-functions $(g_{n,\eta})$, a sequence
$(\cH_n)$ of measurable subsets of $\cH$ and a sequence $(V_{n,\eta})$
of $p\times p$-matrices such that,
\begin{equation}
  \label{eq:mLAN}
  \sup_{h \in K}\sup_{\eta \in \cH_n} \bigg|
    \log \frac{p^{(n)}_{\theta_n(h), \eta}}{p^{(n)}_{\theta_0, \eta}}(X^{(n)})
    - \frac{h^T}{\sqrt{n}} g_{n,\eta} (X^{(n)}) 
    + \frac{1}{2} h^T V_{n,\eta}h \bigg| = o_{P_0}(1),
\end{equation}
for every compact $K \subset \RR^p$, where $\theta_n(h)$ equals
$\theta_0+h/\sqrt{n}$. When we know $\eta_0$, property (\ref{eq:mLAN})
is nothing but the usual parametric LAN expansion, where we set
$\cH_n=\{\eta_0\}$. We refer to (\ref{eq:mLAN}) as the {\it misspecified}
LAN expansion because the base for the expansion is $(\theta_0,\eta)$
while rest-terms go to zero under $P_0$, which corresponds to the
point $(\theta_0,\eta_0)$.

Note that the misspecified
LAN expansion is simpler than the full LAN expansion used in
\cite{shen2002asymptotic, bickel2012semiparametric, castillo2015bernstein}. 
Although the misspecified LAN expansion (\ref{eq:mLAN}) can be applied only
to the adaptive cases, the verification of \eqref{eq:mLAN} is not easy
due to misspecification and the required uniformity of convergence.
LAN expansions have been shown to be valid even under misspecification:
in \cite{kleijn2012bernstein} for example, smoothness in misspecified
parametric models is expressed through a version of local asymptotic
normality under the true distribution of the data, with a likelihood
expansion around points in the model where the Kullback-Leibler (KL)-divergence with
respect to $P_0$ is minimal. In models with symmetric error, the point
of minimal KL-divergence equals exactly $\theta_0$, provided that the
misspecified $\eta$ is close enough to $\eta_0$ in the sense of $\cH_n$.
This allows the usual LAN expansion at $\theta_0$ for fixed $\eta$,
that is, the left-hand side of \eqref{eq:mLAN} is expected to be of
order $o_{P_0}(1)$. By choosing localizations $\cH_n$ appropriately,
the family of score functions $\{\score_{\theta,\eta}:\eta\in\cH_n\}$
is shown to be a Donsker class, which validates \eqref{eq:mLAN} in
models with symmetric errors, where $\score_{\theta,\eta}(x) =
\partial \ell_{\theta,\eta} (x)/\partial \theta$, 
$g_{n,\eta}(X^{(n)}) = \sum_{i=1}^n \score_{\theta_0,\eta}(X_i)$
and $V_{n,\eta} = n^{-1} P_0^{(n)} [g^{\phantom{T}}_{n,\eta} g^T_{n,\eta_0}]$.
The score function is not necessarily the pointwise derivative of the
log-likelihood, but in most examples (including the models considered
in this paper), $g_{n,\eta} = \score_{\theta_0, \eta}^{(n)}$ where
$\score_{\theta, \eta}^{(n)} =  \ell_{\theta, \eta}^{(n)}/ \partial \theta$.
From now on, since it conveys the natural meaning of derivative, we
use the notation  $\score_{\theta_0,\eta}^{(n)}$ instead of $g_{n,\eta}$.

\subsection{The semi-parametric Bernstein-von Mises theorem}

We use a product prior $\Pi=\Pi_\Theta \times \Pi_\cH$ on the Borel
$\sigma$-algebra of $\Theta\times\cH$ and denote the posterior
distribution by $\Pi(\cdot|X^{(n)})$. Note that the misspecified LAN
property gives rise to an expansion
of the log-likelihood that applies only locally in sets $\Theta_n\times\cH_n$,
where $\Theta_n=\{\theta_0 + h/\sqrt{n}: h \in K\}$ (for some compact
$K\in\RR^p$ and appropriate $\cH_n\subset\cH$). So for 
the semi-parametric BvM theorem, the score function 
$\score_{\theta_0, \eta}^{(n)}$ as well as $V_{n,\eta}$ must `behave nicely' on
$\Theta_n\times\cH_n$ and the posterior distribution must concentrate
inside $\Theta_n\times\cH_n$. Technically, these requirements are
expressed by the following two conditions.
For a matrix $A\in\RR^{n_1\times n_2}$, $\|A\|$ represents
the operator norm of $A$, defined as $\sup_{x\neq 0} |Ax|/|x|$, and
if $A$ is a square matrix, $\rho_{\rm min} (A)$ and $\rho_{\rm max}(A)$ denote
the minimum and maximum eigenvalues of $A$, respectively.
\vspace*{1ex}

\noindent{\bf Condition A.} (Equicontinuity and non-singularity)
\bea
  \label{eq:score_conti_condition}
  \sup_{\eta\in\cH_n} \left| \score_{\theta_0,\eta}^{(n)}(X^{(n)})
    - \score_{\theta_0,\eta_0}^{(n)}(X^{(n)}) \right| &=& o_{P_0}(n^{1/2}),
	\\
  \label{eq:V_conti_condition}
  \sup_{\eta\in\cH_n} \|V_{n,\eta} - V_{n,\eta_0}\| &=& o(1),
	\\
  \label{eq:V_positive}
  0 < \liminf_{n\rightarrow\infty}\rho_{\min}(V_{n,\eta_0}) \leq 
    \limsup_{n\rightarrow\infty}\rho_{\max}(V_{n,\eta_0}) &<& \infty.
\eea

\noindent{\bf Condition B.} (Posterior localization)
\bea 
    P_0^{(n)} \Pi\big(\cH_n | X^{(n)}\big) &\rightarrow
    &1, \label{eq:eta_consistency}\\
    P_0^{(n)} \Pi\big(\sqrt{n}|\theta-\theta_0| > M_n | X^{(n)}\big)
    &\rightarrow& 0, ~~ \textrm{for every $M_n \uparrow \infty$}.
	\label{eq:sqrtn_general}
\eea

Conditions like \eqref{eq:score_conti_condition} and
\eqref{eq:V_conti_condition} are to be expected in the context of
semi-parametric estimation (see, \eg, Theorem~25.54 of
\cite{van1996efficient}). Condition \eqref{eq:score_conti_condition}
amounts to \emph{asymptotic equicontinuity} and is implied whenever
scores form a Donsker class, a well-known sufficient
condition in semi-parametric efficiency (see \cite{van1996efficient}).
Condition \eqref{eq:V_conti_condition} is implied whenever the
$L_2(P_0^{(n)})$-norm of the difference between scores at $(\theta_0,\eta)$
and $(\theta_0,\eta_0)$ vanishes as $\eta$ converges to $\eta_0$
in Hellinger distance, \cf\ \eqref{eq:d2_domination}; it controls variations
of the information matrix as $\eta$ converges to $\eta_0$ with $\cH_n$.
Condition \eqref{eq:V_positive} guarantees that the Fisher information
matrix does not develop singularities as the sample size goes to infinity.

Condition \eqref{eq:eta_consistency} formulates a requirement of posterior
consistency in the usual sense, and sufficient conditions are well-known
\cite{schwartz1965bayes, barron1999consistency,
walker2004new, kleijn2013criteria}. Condition \eqref{eq:sqrtn_general}
requires $n^{-1/2}$-rate of convergence rate for the marginal posterior
distribution for the parameter of interest. Though some authors
remark that \eqref{eq:sqrtn_general} appears to be rather too strong
\cite{yang2015semiparametric}, clearly, \eqref{eq:sqrtn_general} is
a {\it necessary condition} (since it follows directly from the
BvM assertion). The proof of condition \eqref{eq:sqrtn_general} is
demanding in a technical sense and forms the most difficult part of
this analysis and most others \cite{bickel2012semiparametric}.

We say the prior $\Pi_\Theta$ is \emph{thick} at $\theta_0$ if it has
a strictly positive and continuous Lebesgue density in a neighborhood of
$\theta_0$. The following theorem states the BvM theorem for semi-parametric
models that are smooth in the sense of the misspecified LAN expansion.
  
\begin{theorem}
\label{thm:BvM_general}
Consider statistical models $\{P^{(n)}_{\theta,\eta}:
\theta \in \Theta, \eta \in \cH\}$ with a product prior
$\Pi=\Pi_\Theta \times \Pi_\cH$. Assume that $\Pi_\Theta$ is thick
at $\theta_0$ and that \eqref{eq:mLAN} as well as Conditions~A
and~B hold. Then,
\begin{equation}
  \label{eq:bvmassertion}
  \sup_B \left|\Pi\big(\sqrt{n}(\theta-\theta_0) \in B |
  X^{(n)}\big) - N_{\Delta_n, V_{n,\eta_0}^{-1}}(B)\right| \rightarrow 0,
\end{equation}
in $P_0^{(n)}$-probability, where,
\[
  \Delta_n= \frac{1}{\sqrt{n}} V_{n,\eta_0}^{-1}
  \score_{\theta_0,\eta_0}^{(n)} (X^{(n)}).
\]
\end{theorem}

\begin{proof}
Note first that \eqref{eq:eta_consistency} implies that
$\Pi_\cH(\cH_n) > 0$ for large enough $n$. Let $\Pi_{\cH_n}$ be the
probability measure obtained by restricting $\Pi_\cH$ to $\cH_n$ and
next re-normalizing, and $\Pi_{\cH_n}(\cdot|X^{(n)})$ be the
corresponding posterior distribution. Then, for any measurable set
$B$ in $\Theta$,
\bean
  \Pi(\theta \in B | X^{(n)}) =
    \Pi(\theta\in B, \eta\in\cH_n | X^{(n)})
    + \Pi(\theta\in B, \eta\in\cH_n^c | X^{(n)})\\
  = \Pi_{\cH_n}(\theta\in B| X^{(n)}) \Pi(\eta\in\cH_n | X^{(n)})
    + \Pi(\theta\in B, \eta\in\cH_n^c | X^{(n)}),
\eean
so we have,
\[
  \sup_B \Big| \Pi(\theta\in B | X^{(n)}) 
    - \Pi_{\cH_n} (\theta\in B | X^{(n)}) \Big| \rightarrow 0,
\]
in $P_0^{(n)}$-probability. Therefore it is sufficient to prove
the BvM assertion with the priors $\Pi_{\cH_n}$.

Particularly, 
\begin{equation}
  \Pi_{\cH_n}(\sqrt{n}|\theta-\theta_0| > M_n | X^{(n)}) 
  =\frac{\Pi(\sqrt{n}|\theta-\theta_0| > M_n, \eta\in\cH_n | X^{(n)})}
	{\Pi(\eta\in\cH_n| X^{(n)})},
\end{equation}
converges to 0 in $P_0^{(n)}$-probability by \eqref{eq:eta_consistency}
and \eqref{eq:sqrtn_general}. Using \eqref{eq:mLAN},
\eqref{eq:score_conti_condition} and \eqref{eq:V_conti_condition},
we obtain,
\begin{equation}
\label{eq:unif_LAN}
  \sup_{h \in K}\sup_{\eta \in \cH_n} \bigg|
    \log \frac{p^{(n)}_{\theta_n(h), \eta}}{p^{(n)}_{\theta_0, \eta}}(X^{(n)})
    - \frac{h^T}{\sqrt{n}} \score_{\theta_0,\eta_0}^{(n)} (X^{(n)}) 
    + \frac{1}{2} h^T V_{n,\eta_0}h \bigg| = o_{P_0}(1),
\end{equation}
for every compact $K \subset \RR^p$. Let,
\[
  b_1(h) = \inf_{\eta\in\cH_n} \frac{ p_{\theta_n(h),\eta}^{(n)}(X^{(n)})}
    {p_{\theta_0,\eta}^{(n)}(X^{(n)})},
  \qquad\text{and}\qquad
  b_2(h) = \sup_{\eta\in\cH_n} \frac{ p_{\theta_n(h),\eta}^{(n)}(X^{(n)})}
    {p_{\theta_0,\eta}^{(n)}(X^{(n)})}.
\]
Then, trivially, we have,
\begin{equation}
  b_1(h) \leq
  \frac{\int {p^{(n)}_{\theta_n(h), \eta}(X^{(n)})} d\Pi_{\cH_n}(\eta)}
       {\int {p^{(n)}_{\theta_0, \eta}(X^{(n)})} d\Pi_{\cH_n}(\eta)}
  \leq b_2(h),
\end{equation}
and the quantity,
\[
  \sup_{h\in K} \Big|b_k(h) 
  - \frac{h^T}{\sqrt{n}} \score_{\theta_0,\eta_0}^{(n)} (X^{(n)}) 
  + \frac{1}{2} h^T V_{n,\eta_0}h\Big|,
\]
is bounded above by the left-hand side of \eqref{eq:unif_LAN} for $k=1,2$.
As a result,
\begin{equation}
  \sup_{h \in K}
  \bigg|\log \frac{\int {p^{(n)}_{\theta_n(h), \eta}(X^{(n)})} d\Pi_{\cH_n}(\eta)}
  {\int {p^{(n)}_{\theta_0, \eta}(X^{(n)})} d\Pi_{\cH_n}(\eta)}
    - \frac{h^T}{\sqrt{n}} \score^{(n)}_{\theta_0,\eta_0}(X^{(n)}) 
    + \frac{1}{2}h^T V_{n,\eta_0}h \bigg| = o_{P_0}(1),
\end{equation}
because $|c_2| \leq |c_1|\vee|c_3|$ for all real numbers
$c_1, c_2$ and $c_3$ with $c_1 \leq c_2 \leq c_3$. The remainder of
the proof is (almost) identical to the proof for parametric models
\cite{le1990asymptotic,kleijn2012bernstein}, replacing the parametric
likelihood by $\theta \mapsto \int {p^{(n)}_{\theta, \eta}(X^{(n)})}
d\Pi_{\cH_n}(\eta)$ as in \cite{bickel2012semiparametric}, details
of which can be found in Theorem 3.1.1 of \cite{chae2015semiparametric}.
\end{proof}

\section{Semi-parametric BvM for linear regression models}
\label{sec:regression}

Let $\cH$ be the set of all continuously differentiable densities
$\eta$ defined on $\DD=(-r,r)$ (for some $r\in (0,\infty]$) such
that $\eta(x) > 0$ and $\eta(x) = \eta(-x)$ for every $x\in\DD$.
Equip $\cH$ with the Hellinger metric.
We consider a model for data satisfying,
\begin{equation}
\label{eq:lm}
  X_i = \theta^T Z_i + \epsilon_i, \quad\text{for $i=1, \ldots, n$},
\end{equation}
where $Z_i$'s are $p$-dimensional non-random covariates and the
errors $\epsilon_i$ are assumed to form an \iid\ sample from a distribution
with density $\eta\in\cH$. We prove the BvM theorem for
the regression coefficient $\theta$. 

Let $P_{\theta,\eta,i}$ denote the
probability measure with density $x \mapsto \eta(x-\theta^T Z_i)$
and $\score_{\theta,\eta,i} = \partial \ell_{\theta,\eta,i}/\partial \theta$.
Also let $P_\eta$ be the probability measure with density $p_\eta = \eta$
and $s_\eta(x) = -\partial \ell_\eta(x)/\partial x$. Let $P_{\theta,\eta}^{(n)}$
represent the product measure
$P_{\theta,\eta,1} \times \cdots \times P_{\theta,\eta,n}$
and let $\score^{(n)}_{\theta,\eta} = \sum_{i=1}^n \score_{\theta,\eta,i}$.
With slight abuse of notation, we treat $p_{\theta,\eta,i},
\ell_{\theta,\eta,i}$ and $\score_{\theta,\eta,i}$ as either functions
of $x$ or the corresponding random variables when they are evaluated
at $x=X_i$. For example, $\score_{\theta,\eta,i}$ represents either the
function $x \mapsto \score_{\theta,\eta,i}(x): \DD \mapsto \RR^p$ or
the random vector $\score_{\theta,\eta,i}(X_i)$. We treat
$p^{(n)}_{\theta,\eta}, \ell^{(n)}_{\theta,\eta}$ and
$\score^{(n)}_{\theta,\eta}$ similarly.

Let $\theta_0\in\Theta$ and $\eta_0\in\cH$ be the true regression
coefficient and error density in the model (\ref{eq:lm}). Define
specialized KL-balls in $\Theta\times\cH$ of the form,
\begin{equation}
\label{eq:B_n-def}
  B_n(\epsilon) = \Big\{ (\theta,\eta): 
  \sum_{i=1}^n K(p_{\theta_0,\eta_0,i}, p_{\theta,\eta,i}) \leq n\epsilon^2,
  \sum_{i=1}^n V(p_{\theta_0,\eta_0,i}, p_{\theta,\eta,i})\leq C_2n\epsilon^2\Big\},
\end{equation}
where $K(p_1,p_2) = \int\log (p_1/p_2) dP_1$, $V(p_1,p_2) = \int
(\log(p_1/p_2) - K(p_1,p_2))^2 dP_1$, and
$C_2$ is some positive constant (see \cite{ghosal2007convergence}).
Define the mean Hellinger distance $h_n$ on $\Theta\times\cH$ by,
\begin{equation}
\label{eq:h_n-def}
  h^2_n\big((\theta_1, \eta_1), (\theta_2,\eta_2)\big) 
    = \frac{1}{n} \sum_{i=1}^n h^2(p_{\theta_1, \eta_1,i}, p_{\theta_2,\eta_2,i}).
\end{equation}
Let $v_\eta = P_{\eta_0} [s_\eta s_{\eta_0} ]$ and,
\be \label{eq:V_n_eta-def}
  V_{n,\eta} = \frac{1}{n} P_0^{(n)} \big[\score^{(n)}_{\theta_0,\eta}
    \score^{(n) T}_{\theta_0,\eta_0}\big]. 
\ee
It is easy to see that $V_{n,\eta}= v_\eta \bZ_n$,
where $\bZ_n = n^{-1} \sum_{i=1}^n Z_i Z_i^T$.

We say that a sequence of real-valued stochastic processes
$\{Y_n(t): t \in T\}$, ($n\geq1$), is \emph{asymptotically tight} if
it is asymptotically tight in the space of bounded functions on $T$
with the uniform norm \cite{van1996weak}. A vector-valued stochastic
process is asymptotic tight if each of its components is
asymptotically tight.

\begin{theorem}
\label{th:main-linear}
Suppose that $\sup_{i\geq 1}|Z_i|\leq L$ for some constant $L > 0$,
$\liminf_n \rho_{\rm min}(\bZ_n) > 0$ and $v_{\eta_0}>0$. The prior
for $(\theta,\eta)$ is a product $\Pi=\Pi_\Theta\times \Pi_{\cH}$,
where $\Pi_\Theta$ is thick at $\theta_0$. Suppose also that there
exist an $N\geq1$, a sequence $\epsilon_n \rightarrow 0$ with
$n\epsilon_n^2\rightarrow\infty$, and partitions
$\Theta=\Theta_{n,1} \cup \Theta_{n,2}$ and $\cH=\cH_{n,1}\cup
\cH_{n,2}$ such that $\eta_0\in\cH_{n,1}$ and
\be
  \begin{split}\label{eq:rate_reg}
  \log N(\epsilon_n/36, \Theta_{n,1}\times\cH_{n,1}, h_n)
    &\leq n\epsilon_n^2, 
    \\
  \log\Pi\big(B_n(\epsilon_n)\big) 
    &\geq -\frac{1}{4}n\epsilon_n^2, 
    \\
  \log \big(\Pi_\Theta(\Theta_{n,2}) + \Pi_\cH(\cH_{n,2})\big)
    &\leq -\frac{5}{2}n\epsilon_n^2, 
  \end{split}
\ee
for all $n\geq N$. For some $\overline M_n\uparrow\infty$, with
$\epsilon_n \overline M_n \rightarrow0$, let $\cH_n 
= \{\eta\in\cH_{n,1}: h(\eta,\eta_0) < \overline M_n\epsilon_n\}$
and assume that there exist a continuous $L_2(P_{\eta_0})$-function
$Q$ and an $\epsilon_0 > 0$ such that,
\begin{equation}\label{eq:Lipsc_cond_reg}
	\sup_{|y|<\epsilon_0}\sup_{\eta\in\cH^N}\left|
	\frac{\ell_\eta(x+y)- \ell_\eta(x)}{y}\right| 
	\vee \left|\frac{s_\eta(x+y)- s_\eta(x)}{y}\right| \leq  Q(x),
\end{equation}
where $\cH^N = \cup_{n=N}^\infty\cH_n$.
Furthermore, assume that the sequence of stochastic processes,
\begin{equation}\label{eq:score_process_reg}
	\bigg\{\frac{1}{\sqrt{n}}\Big(\score_{\theta,\eta}^{(n)}
	- P_0^{(n)} \score_{\theta,\eta}^{(n)}\Big):
	|\theta-\theta_0|<\epsilon_0, \eta\in\cH^N \bigg\},
\end{equation}
indexed by $(\theta,\eta)$ is asymptotically tight. Then the assertion
of the BvM theorem~\ref{thm:BvM_general} holds for $\theta$.
\end{theorem}

Since the observations are not \iid, we consider the mean Hellinger
distance $h_n$ as in \cite{ghosal2007convergence}. Conditions
\eqref{eq:rate_reg} are required for the
convergence rate of $h_n\big((\theta,\eta),(\theta_0,\eta_0)\big)$
to be $\epsilon_n$, which in turn implies that the convergence rates
of $|\theta-\theta_0|$ and $h(\eta,\eta_0)$ are $\epsilon_n$ (\cf\
Lemma \ref{lem:reg_consistency}).
In fact, we only need to prove \eqref{eq:rate_reg} with arbitrary rate $\epsilon_n$ because the so-called no-bias condition $\sup_{\eta\in\cH_n} P_0 \score_{\theta_0, \eta}^{(n)} = o_{P_0}(n^{-1/2})$ holds trivially by the symmetry, which plays an important role to prove \eqref{eq:mLAN}-\eqref{eq:V_conti_condition} as in frequentist literature (see Chapter~25 of \cite{van1998asymptotic}).
Condition \eqref{eq:Lipsc_cond_reg}, which is technical in nature,
is easily satisfied.
For a random design, \eqref{eq:score_process_reg}
is asymptotically tight if and only if the class of score functions
forms a Donsker class,
and sufficient conditions for the latter are
well established in empirical process theory. Since observations
are not \iid\ due to the non-randomness of covariates,
\eqref{eq:score_process_reg} does not converge in distribution to
a Gaussian process. Here, asymptotic tightness of
\eqref{eq:score_process_reg} merely assures that the supremum of
its norm is of order $O_{P_0}(1)$. Asymptotic tightness holds under
a finite bracketing integral condition (where the definition of
the bracketing number is extended to non-\iid\ observations in a
natural way). For sufficient conditions for asymptotic tightness with
non-\iid\ observations, readers are referred to section~2.11 of
\cite{van1996weak}.

We prove Theorem \ref{th:main-linear} by checking the misspecified
LAN condition as well as Conditions A and B, whose proofs are sketched
in the three following subsections respectively. Detailed proofs are
provided in the appendix.

\subsection{Proof of Misspecified LAN}

Note that $P_0^{(n)} \score_{\theta_0,\eta}^{(n)} = 0$ for every
$\eta \in \cH$ by the symmetry of $\eta$ and $\eta_0$.
This enables writing the left-hand side of \eqref{eq:mLAN} as,
\[
  \log \frac{p^{(n)}_{\theta_n(h), \eta}}{p^{(n)}_{\theta_0, \eta}}(X^{(n)})
    - \frac{h^T}{\sqrt{n}} \score_{\theta_0,\eta}^{(n)} (X^{(n)}) 
    + \frac{1}{2} h^T V_{n,\eta}h = A_n(h,\eta) + B_n(h,\eta),
\]
where,
\be
  \begin{split} \label{eq:AB_n}
  A_n(h,\eta) &= \left(\ell_{\theta_n(h),\eta}^{(n)}
    - \ell_{\theta_0,\eta}^{(n)} - \frac{h^T}{\sqrt{n}}
    \score_{\theta_0,\eta}^{(n)} \right)^o,
	\\
  B_n(h,\eta) &= P_0^{(n)} \Big( \ell^{(n)}_{\theta_n(h),\eta}
    - \ell^{(n)}_{\theta_0,\eta}\Big) + \frac{1}{2} h^T V_{n, \eta}h.
  \end{split}
\ee
It suffices to prove that $A_n(h,\eta)$ and $B_n(h,\eta)$ converge to zero 
uniformly over $h\in K$ and $\eta\in\cH^N$, in
$P_0^{(n)}$-probability, for every compact set $K$.

Note that $A_n(h,\eta)$ is equal to,
\[
	\frac{h^T}{\sqrt{n}}\int_0^1\left(\score_{\theta_n(th),\eta}^{(n)}
-\score_{\theta_0,\eta}^{(n)}\right)^o dt,
\]
by Taylor expansion, so for a compact set $K$, we have,
\be\label{eq:tight_process}
	\sup_{h\in K}\sup_{\eta\in\cH^N} |A_n(h,\eta)| \lesssim
	\sup_{h\in K}\sup_{\eta\in\cH^N} \bigg|\frac{1}{\sqrt{n}}
\left(\score_{\theta_n(h),\eta}^{(n)}
  -\score_{\theta_0,\eta}^{(n)}\right)^o
\bigg|.
\ee
For fixed $h\in K$ and $\eta\in\cH^N$,
$n^{-1/2} \left(\score_{\theta_n(h),\eta}^{(n)}
-\score_{\theta_0,\eta}^{(n)}\right)^o$
converges to zero in probability because its mean is zero and its
variance is bounded by,
\[
  \begin{split}
  \frac{1}{n}\sum_{i=1}^n P_0 &\left|
    \score_{\theta_n(h),\eta,i}- \score_{\theta_0,\eta,i}\right|^2\\
  &\lesssim \frac{1}{n} \sum_{i=1}^n P_0 \left|
  s_\eta\left(X_i - \theta_n(h)^T Z_i\right) 
  - s_\eta\left(X_i - \theta_0^T Z_i\right)\right|^2\\
  &\leq \frac{1}{n} \sum_{i=1}^n
 |(\theta_n(h)-\theta_0)^T Z_i |^2 \cdot P_{\eta_0} Q^2
   \lesssim \frac{P_{\eta_0} Q^2}{n},
  \end{split}
\]
which converges to zero as $n\rightarrow \infty$.
In turn, the pointwise convergence of
$n^{-1/2} \left(\score_{\theta_n(h),\eta}^{(n)}
  -\score_{\theta_0,\eta}^{(n)}\right)^o$
to zero implies uniform convergence to zero of the right-hand
side of \eqref{eq:tight_process}, since \eqref{eq:score_process_reg}
is asymptotically tight. Thus the supremum of $|A_n(h,\eta)|$ over
$h\in K$ and $\eta\in \cH^N$ is of order $o_{P_0}(1)$.

For $B_n(h,\eta)$, we prove in Section \ref{subsubsec:1} that,
\begin{equation}
  \label{eq:unif_quad_reg}
  \sup_{\eta\in\cH^N} \left| \frac{1}{n}P_0^{(n)}
  \Big( \ell^{(n)}_{\theta,\eta} - \ell^{(n)}_{\theta_0,\eta}\Big)
  + \frac{1}{2} (\theta-\theta_0)^T V_{n,\eta} (\theta-\theta_0) \right|
  = o(|\theta-\theta_0|^2),
\end{equation}
as $\theta \rightarrow \theta_0$. 
Consequently, the supremum of $B_n(h,\eta)$ over $h \in K$ and
$\eta\in \cH^N$ converges to zero.
\qed

\subsection{Proof of Condition A}

For given $\eta,\eta_0$, let $d_2$ be the metric on $\cH$ defined by,
\be\label{eq:d2_def}
	d_2^2(\eta,\eta_0) = P_{\eta_0}(s_\eta - s_{\eta_0})^2.
\ee
In Section \ref{subsubsec:2}, it is shown
that,
\be\label{eq:d2_domination}
	\lim_{n\rightarrow\infty} \sup_{\eta\in\cH_n} d_2(\eta,\eta_0) = 0.
\ee
Let $a\in\RR^p$ be a non-zero vector and let $\sigma_n^2 = a^T\bZ_n a$.
Because $\rho_{\rm min}(\bZ_n)$ is bounded away from zero in the tail
by assumption, $\sigma_n^2$ is bounded away from zero for large enough
$n$, and so the scaled process,
\be\label{eq:score_theta0}
	\bigg\{\frac{a^T}{\sqrt{n}
          \sigma_n}\Big(\score_{\theta_0,\eta}^{(n)}
  - P_0^{(n)} \score_{\theta_0,\eta}^{(n)}\Big): \eta\in\cH^N \bigg\},
\ee
is asymptotically tight by the asymptotic tightness of
\eqref{eq:score_process_reg}. Furthermore, it converges weakly (in
the space of bounded functions with the uniform norm) to a tight
Gaussian process because it coverges marginally to a Gaussian
distribution by the Lindberg-Feller theorem. To see this, the variance
of \eqref{eq:score_theta0} for fixed $\eta$ is equal to $P_{\eta_0}
s_\eta^2$ for every $n$. In addition, 
\[
  \begin{split}
\frac{1}{n\sigma_n^2} \sum_{i=1}^n
  P_0 &|a^T\score_{\theta_0,\eta,i}|^2
    1_{\{ |a^T \score_{\theta_0,\eta,i}| > \sqrt{n} \sigma_n\epsilon\}}\\
  &= \frac{1}{n\sigma_n^2} \sum_{i=1}^n |a^T Z_i|^2
    P_{\eta_0} s_\eta^2 1_{\{|s_\eta| \geq
      \sqrt{n}\epsilon \sigma_n/|a^T Z_i|\}}\\
  &\lesssim \frac{1}{n} \sum_{i=1}^n
    P_{\eta_0} s_\eta^2 1_{\{|s_\eta| \geq
      \sqrt{n}\epsilon \sigma_n/|a^T Z_i|\}}
  \leq P_{\eta_0} s_\eta^2 1_{\{|s_\eta| \gtrsim \sqrt{n}\epsilon\}}
  = o(1),
  \end{split}
\]
for every $\epsilon > 0$ and large enough $n$. By the weak convergence
of \eqref{eq:score_theta0} to a tight Gaussian process,
\eqref{eq:score_theta0} is uniformly $d_2$-equicontinuous in
probability (see Section 1.5 of \cite{van1996weak}), because,
\[
	P_0\bigg[\frac{a^T}{\sqrt{n}\sigma_n}
\Big(\score^{(n)}_{\theta_0,\eta}- \score^{(n)}_{\theta_0,\eta'}\Big) \bigg]^2
	= \frac{1}{n\sigma_n^2}
\sum_{i=1}^n a^T Z_i Z_i^T a P_{\eta_0} \big(s_{\eta}-s_{\eta'}\big)^2
	= d^2_2(s_{\eta}, s_{\eta'}),
\]
for every $n \geq 1$. Since $P_0^{(n)} \score_{\theta_0,\eta}^{(n)} =
0$ for every $\eta\in\cH^N$, by the definition of asymptotic
equicontinuity, we have,
\[
  \sup\Biggl\{\biggl|\frac{a^T(\score_{\theta_0,\eta}^{(n)}
    - \score_{\theta_0,\eta_0}^{(n)})}{\sigma_n} \biggr|:
      d_2(\eta,\eta_0) < \delta_n,\,\eta\in\cH^N\Biggr\}
  = o_{P_0}(n^{1/2}),
\]
for every $\delta_n \downarrow 0$.
Since $\sigma_n$ is bounded away from zero for large $n$ and $a$ is
arbitrary, \eqref{eq:d2_domination} implies
\eqref{eq:score_conti_condition}.

For \eqref{eq:V_conti_condition}, note that,
\[
  \|V_{n,\eta} - V_{n,\eta_0}\|
  = \|(v_\eta-v_{\eta_0})\bZ_n\| 
  = |v_\eta-v_{\eta_0}|\cdot \|\bZ_n\|
  = \rho_{\max} (\bZ_n) \cdot |v_\eta-v_{\eta_0}|,
\]
and $\limsup_n \rho_{\rm max}(\bZ_n) < \infty$ because covariates are bounded.
Since,
\[
  |v_\eta-v_{\eta_0}| = |P_{\eta_0} (s_\eta - s_{\eta_0})s_{\eta_0}|
  \lesssim d_2(\eta,\eta_0),
\]
by the Cauchy-Schwartz inequality, we have $\|V_{n,\eta} -
V_{n,\eta_0}\| \lesssim d_2(\eta,\eta_0)$,
and thus \eqref{eq:d2_domination} implies \eqref{eq:V_conti_condition}.

Finally, since $v_{\eta_0} > 0, \liminf_n \rho_{\rm min}(\bZ_n) > 0$
and $\sup_{i\geq 1} |Z_i| \leq L$, \eqref{eq:V_positive} holds
trivially because $V_{n,\eta} = v_\eta \bZ_n$.
\qed

\subsection{Proof of Condition B}

We need the following lemma, the proof of which is found in Section
\ref{subsubsec:3}.
\begin{lemma}
\label{lem:reg_consistency} Under the conditions in
Theorem \ref{th:main-linear}, there exists $K>0$ such that for
every sufficiently small $\epsilon > 0$ and large enough $n$,
$\eta\in\cH_n$ and $h_n\big((\theta,\eta),(\theta_0,\eta_0)\big)
< \epsilon$ imply $|\theta-\theta_0| < K\epsilon$ and
$h(\eta,\eta_0) < K\epsilon$. 
\end{lemma}
Under the conditions in Theorem \ref{th:main-linear},
it is well known (see Theorem~4 of \cite{ghosal2007convergence}) that,
\be
\label{eq:ghoshal}
  P_0^{(n)} \Pi\Big((\theta,\eta) \in \Theta_{n,1}\times\cH_{n,1}: 
    h_n\big((\theta,\eta),(\theta_0,\eta_0)\big) \leq M_n \epsilon_n 
    \big| X^{(n)}\Big) \rightarrow 1,
\ee
for every $M_n\rightarrow \infty$.
Thus Lemma \ref{lem:reg_consistency} implies
\eqref{eq:eta_consistency}.

For \eqref{eq:sqrtn_general}, let $\epsilon > 0$ be a sufficiently
small constant and $(M_n)$ be a real sequence such
that $M_n \rightarrow \infty$ and $M_n/\sqrt{n}\rightarrow 0$. 
Also, let $\Theta_n = \{\theta\in\Theta_{n,1}: M_n/\sqrt{n}
< |\theta-\theta_0| \leq \epsilon\}$. Since,
\[
  \begin{split}
  \Pi\bigl(&\sqrt{n}|\theta-\theta_0| > M_n \bigm| X^{(n)}\bigr)\\
  &= \Pi\bigl(|\theta-\theta_0| > \epsilon \bigm| X^{(n)}\bigr)
  + \int \Pi\bigl(\theta\in\Theta_n\bigm|\eta,X^{(n)}\bigr)
    d\Pi(\eta|X^{(n)})\\
  &\leq \Pi\bigl(|\theta-\theta_0| > \epsilon \bigm| X^{(n)}\bigr)
  + \sup_{\eta\in\cH_n}\Pi\bigl(\theta \in \Theta_n \bigm| \eta, X^{(n)}\bigr)
  + \Pi(\eta\in\cH_n^c| X^{(n)}),
  \end{split}
\]
and $\Pi\bigl(|\theta-\theta_0| > \epsilon \bigm| X^{(n)}\bigr)
\vee \Pi(\eta\in\cH_n^c| X^{(n)})$ converges to 0 in $P_0^{(n)}$-probability 
due to (\ref{eq:ghoshal}) with Lemma \ref{lem:reg_consistency},
it suffices to show that,
\be
\label{eq:unif_conv_cond_posterior}
  \sup_{\eta\in\cH_n}\Pi\bigl(\theta \in \Theta_n \bigm| \eta, X^{(n)}\bigr)
  \rightarrow 0,
\ee
in $P_0^{(n)}$-probability.
Note that,
\[
  \Pi\bigl(\theta \in \Theta_n \bigm| \eta, X^{(n)}\bigr)
  = \frac{\int_{\Theta_n} p^{(n)}_{\theta,\eta} / p^{(n)}_{\theta_0,\eta}(X^{(n)})
    \,d\Pi_\Theta(\theta)}
    {\int p^{(n)}_{\theta,\eta} / p^{(n)}_{\theta_0,\eta}(X^{(n)})
    \,d\Pi_\Theta(\theta)},
\]
by Bayes's rule.
In Section \ref{subsubsec:4}, we prove that 
we can choose $C > C_1 > 0$ and $C_2 > 0$ such that,
\be
\label{eq:AB_n_sqrtn}
P_0^{(n)}(A_n\cap B_n)\rightarrow 1,
\ee
where,
\be
\begin{split}
  A_n &= \left\{\inf_{\eta\in\cH_n} \int_\Theta
    \frac{p^{(n)}_{\theta, \eta}}{p^{(n)}_{\theta_0, \eta}}
    \,d\Pi_\Theta(\theta) \geq C_2 
    \left(\frac{M_n}{\sqrt{n}}\right)^p e^{-C_1 M_n^2} \right\},\\
  B_n &= \left\{\sup_{M_n < |h| < \epsilon \sqrt{n}}
    \sup_{\eta\in\cH_n} \frac{p^{(n)}_{\theta_n(h), \eta}}
    {p^{(n)}_{\theta_0, \eta}} e^{C |h|^2}\leq 1\right\}.
\end{split}
\ee
The remainder of the proof is similar to that of \cite{lecam1973convergence}.
Let $\Omega_n=A_n\cap B_n$, 
\[
  \Theta_{n,j} = \{\theta_n(h) \in \Theta_n: j M_n \leq |h| < (j+1) M_n\},
\]
and $J$ be the minimum among $j$'s satisfying $(j+1)M_n/\sqrt{n} > \epsilon$.
Since $\Pi_\Theta$ is thick at $\theta_0$ and $\epsilon$ is sufficiently small,
\[
  \Pi_\Theta(\Theta_{n,j}) \leq D \cdot \big((j+1)M_n/\sqrt{n}\big)^p,
\]
for some constant $D > 0$. Then on $\Omega_n$, 
\bean
  \sup_{\eta\in\cH_n} \Pi(\theta \in \Theta_n | \eta, X^{(n)})
  &\leq& \frac{e^{C_1 M_n^2}}{C_2 (M_n/\sqrt{n})^p} 
  \sup_{\eta\in\cH_n} \int_{\Theta_n} 
  \frac{p^{(n)}_{\theta, \eta}}{p^{(n)}_{\theta_0, \eta}} d\Pi_\Theta(\theta)\\
  &\leq& \frac{e^{C_1 M_n^2}}{C_2 (M_n/\sqrt{n})^p}
  \sum_{j=1}^J \Pi_\Theta(\Theta_{n,j}) \sup_{\theta \in \Theta_{n,j}}
  \sup_{\eta\in\cH_n}\frac{p^{(n)}_{\theta, \eta}}{p^{(n)}_{\theta_0, \eta}}.
\eean
Since $\sup_{\theta \in \Theta_{n,j}} \sup_{\eta\in\cH_n}
p^{(n)}_{\theta, \eta}/p^{(n)}_{\theta_0, \eta}\leq \exp(-C j^2 M_n^2)$
on $\Omega_n$, we have,
\begin{equation}
  \label{eq:series_to0}
  \sup_{\eta\in\cH_n} \Pi(\theta \in \Theta_n | \eta, X^{(n)}) 
  \leq C_2^{-1} D e^{C_1M_n^2} \sum_{j=1}^J (j+1)^p e^{-Cj^2M_n^2},
\end{equation}
on $\Omega_n$. Since $C > C_1$, the term on the right-hand side of
\eqref{eq:series_to0} converges to zero as $n \rightarrow \infty$, so
we conclude that \eqref{eq:unif_conv_cond_posterior} holds.

\subsection{Examples} \label{ssec:ex-reg}

Conditions in Theorem \ref{th:main-linear} depend particularly
on the choice of prior for the nuisance parameter $\eta$. In this
subsection, we verify the conditions in Theorem \ref{th:main-linear}
for two priors: a symmetric Dirichlet mixture of normal
distributions and a random series prior on a smoothness class. 
For a given density $p$ on $\DD$, its \emph{symmetrization} $\bar p$
is defined by $\bar p = (p + p^-) / 2$, where $p^-(x) = p(-x)$ for all
$x\in\DD$. We can construct a prior on $\cH$ by putting a prior on
$p\in\widetilde \cH$ and symmetrizing it, where $\widetilde \cH$ is the
set of every density on $\DD$ whose symmetrization belongs to $\cH$.
Obviously, we have $\cH \subset \widetilde \cH$. In this subsection,
let $\Pi_{\widetilde\cH}$ be a probability measure on $\widetilde\cH$
and $\Pi_\cH$ be the corresponding probability measure on $\cH$.
Hellinger entropy bounds and prior concentration rates around KL
neighborhoods are well known for various choices of
$\Pi_{\widetilde\cH}$, so the following lemma is useful to prove
\eqref{eq:rate_reg}.

\begin{lemma}\label{lem:conv_rate}
For a subset $\widetilde \cH_0$ of $\widetilde\cH$ containing
$\eta_0$, suppose that there exists a function $\widetilde Q$
such that $\sup_{\eta\in\widetilde \cH_0} P_{\eta} \widetilde Q^2 <
\infty$, and for every $x$ and sufficiently small $y$,
\be
  \label{eq:lipschitz_bd_lem}
  \sup_{\eta\in \widetilde \cH_0}
  \left|\frac{\log \eta(x+y)-\log\eta(x)}{y}\right| \leq \widetilde Q(x).
\ee
Furthermore, assume that for large enough $n$,
\be
  \begin{split}
  \label{eq:rate_condition_general}
  \log N(\widetilde \epsilon_n, \widetilde \cH_{n,1}, h)
    &\lesssim n\widetilde\epsilon_n^2,\\
  \log \Pi_{\widetilde\cH}\big(\{\eta\in\widetilde\cH:
    K(\eta_0,\eta) \leq \widetilde\epsilon_n^2, V(\eta_0, \eta)
  \leq \widetilde\epsilon_n^2\}\big)
    &\gtrsim -n\widetilde\epsilon_n^2,\\
  \log \Pi_{\widetilde\cH} (\widetilde\cH_{n,2})
    &\leq -\frac{5}{2} n\widetilde\epsilon_n^2 M_n^2,
  \end{split}
\ee
for some partition $\widetilde\cH = \widetilde\cH_{n,1}
\cup \widetilde\cH_{n,2}$  with $\eta_0 \in \widetilde \cH_{n,1}
\subset \widetilde\cH_0$ and sequences
$\widetilde\epsilon_n\rightarrow0$, $M_n\rightarrow\infty$ with
$\widetilde\epsilon_n \gtrsim n^{-1/2}\log n$. If $\Theta$ is compact
and $\sup_{i\geq 1} |Z_i| \leq L$, then, for any $\Pi_\Theta$ that is
thick at $\theta_0$, the product prior $\Pi_\Theta\times\Pi_\cH$
satisfies \eqref{eq:rate_reg} with some $\cH_{n,1} \subset \cH_0$,
$\Theta_{n,1}=\Theta$ and $\epsilon_n = \widetilde\epsilon_n M_n$,
where $\cH_0$ is the set of symmetrizations of $p \in \widetilde\cH_0$.
\end{lemma}
\begin{proof}
For any pair of densities $p$ and $q$ on $\DD$, it is shown in
Section \ref{subsubsec:5} that,
\be
  \label{eq:symm_bd}
  \begin{split}
  h(\bar p, \bar q) &\leq \sqrt 2 h(p,q),\quad
  K(\bar p, \bar q) \leq K(\bar p,q),\\
  V(\bar p, \bar q) &\leq 4\big( V(\bar p,q) + K^2 (\bar p,q)\big),
  \end{split}
\ee
It is also shown in Section \ref{subsubsec:6} that there exist constants $C >0$ and $\epsilon > 0$ such that,
\be
\begin{split}\label{eq:prod_bd}
  h(p_{\theta_1,\eta_1,i}, p_{\theta_2,\eta_2,i})
  &\leq C \big(|\theta_1-\theta_2| + h(\eta_1,\eta_2)\big),
  \\
  K(p_{\theta_0,\eta_0,i}, p_{\theta,\eta,i})
  &\leq C \big(|\theta-\theta_0| + K(\eta_0,\eta)\big),
  \\
  V(p_{\theta_0,\eta_0,i}, p_{\theta,\eta,i})
  &\leq C \big(|\theta-\theta_0|^2 + V(\eta_0,\eta) + K^2(\eta_0,\eta)\big),
\end{split}\ee
for all $\eta_1,\eta_2,\eta\in \cH_0$, $i \geq 1$ and
$\theta_1,\theta_2,\theta$ with
$|\theta_1-\theta_2| \vee |\theta-\theta_0| < \epsilon$.

Let $\cH_{n,1}$ be the set of symmetrizations of
$p \in \widetilde\cH_{n,1}$. By the first inequalities of
\eqref{eq:symm_bd} and \eqref{eq:prod_bd}, there is a $C_1 >0$
such that for large enough $n$,
\[
  \begin{split}
  \log N(\epsilon_n/36,& \Theta_{n,1}\times\cH_{n,1}, h_n)\\
  &\lesssim \log N(C_1 \epsilon_n, \Theta_{n,1}, |\cdot|)
    + \log N(C_1\epsilon_n, \widetilde\cH_{n,1}, h)\\
  &\lesssim \log\epsilon_n^{-1} + n\widetilde\epsilon_n^2\leq n\epsilon_n^2,
\end{split}
\]
where the last inequality follows from
$\epsilon_n > \widetilde\epsilon_n \gtrsim n^{-1/2} \log n$, so
$\log\epsilon_n^{-1} \leq \log ( n^{1/2}/\log n) 
\leq \log n = o(n\epsilon_n^2)$. The second and third inequalities
of \eqref{eq:symm_bd} and \eqref{eq:prod_bd}, with $p=\bar p =
\eta_0$, imply that there exists a constant $C_2 > 0$ such that,
\bean
	\log \Pi(B_n(\epsilon_n)) 
	&\geq& \log \Pi_{\widetilde\cH}\big(\{\eta\in\widetilde\cH:
  K(\eta_0,\eta) \leq  C_2 \epsilon_n^2, V(\eta_0, \eta)
  \leq C_2 \epsilon_n^2\}\big)
	\\
	&& \qquad + \log \Pi_\Theta (\{\theta: |\theta-\theta_0|
\leq C_2 \epsilon_n^2\})
	\\
	&\geq& \log \Pi_{\widetilde\cH}\big(\{\eta\in\widetilde\cH:
  K(\eta_0,\eta) \leq  \widetilde \epsilon_n^2, V(\eta_0, \eta)
  \leq \widetilde \epsilon_n^2\}\big)
	\\
	&& \qquad + \log \Pi_\Theta (\{\theta: |\theta-\theta_0|
\leq \widetilde\epsilon_n^2\})
	\\
	&\gtrsim& -n\widetilde\epsilon_n^2  + \log (\widetilde\epsilon_n^2) 
	\gtrsim -n\widetilde\epsilon_n^2 - \log n \geq -n\epsilon_n^2/4,
\eean
for large enough $n$.
Finally, since,
\bean
	\log \big(\Pi_\cH(\cH_{n,2})\big) \leq \log
  \big(\Pi_{\widetilde\cH} (\widetilde\cH_{n,2})\big)
  \leq -\frac{5}{2} n \epsilon_n^2,
\eean
the proof is complete.
\end{proof}

\subsubsection{Symmetric Dirichlet mixtures of normal distributions}
\label{sssec:dpm}

We consider a symmetrized Dirichlet process mixture of normal densities
for the prior of $\eta$.  Dirichlet process mixture priors are 
popular and the asymptotic behavior of the
posterior distribution is well-studied. A random density $\eta$ is said
to follow a Dirichlet process mixture of normal densities
\cite{lo1984class} if $\eta(x) = \int\phi_\sigma(x-z)dF(z,\sigma)$, 
where $F \sim {\rm DP}(\alpha, H)$ and $\phi_\sigma$ is the density of
the normal distribution with mean 0 and variance $\sigma^2$. Here,
${\rm DP}(\alpha, H)$ denotes the Dirichlet process 
with precision $\alpha>0$ and mean probability measure $H$ on
$\RR \times (0,\infty)$ \cite{ferguson1973bayesian}.


For given positive numbers $\sigma_1, \sigma_2$, and $M$ with
$\sigma_1 < \sigma_2$, let $\cF$ be the set of all distribution
functions supported on $[-M,M]\times[\sigma_1,\sigma_2]$,
and  let $\widetilde\cH_0$ be the set of all densities $\eta$ on $\RR$ 
of the form $\eta(x) = \int\phi_\sigma(x-z)dF(z,\sigma)$ for $F\in \tilde{\cF}$.
Then it is easy to show that $\cH_0$, the symmetrization of $\widetilde\cH_0$, 
is the set of all $\eta\in\widetilde\cH_0$, where $F\in\cF$ 
with $dF(z,\sigma) = dF(-z,\sigma)$. If $F \sim {\rm DP}(\alpha, H)$,
where $H$ has a positive and continuous density supported on 
$[-M,M]\times[\sigma_1,\sigma_2]$, the corresponding Dirichlet
process mixture prior and its symmerization, denoted by
$\Pi_{\widetilde\cH}$ and $\Pi_\cH$, respectively, have full support
on $\widetilde\cH_0$ and $\cH_0$ relative to the Hellinger topology.

\begin{corollary}\label{cor:ex-dpm}
Suppose that $\sup_{i\geq 1}|Z_i| \leq L$ and
$\liminf_n \rho_{\min}(\bZ_n) > 0$.
With the symmetrized Dirichlet process mixture prior described
above for $\eta$, the BvM
theorem holds for the linear regression model provided that
$\eta_0\in\cH_0$ and that 
$\Pi_\Theta$ is compactly supported and thick at $\theta_0$.
\end{corollary}
\begin{proof}
We may assume that $\Theta$ is compact, and let $\Theta_{n,1} = \Theta$.
It is trivial that $v_{\eta_0} > 0$.
The first and second derivatives of the map $x \mapsto \ell_\eta(x)$
are of orders $O(x)$ and $O(x^2)$, respectively, as
$x \rightarrow \infty$ (see lemma 3.2.3 of
\cite{chae2015semiparametric} for details),
and both bounds can be chosen independently of $\eta$. 
Consequently, condition \eqref{eq:Lipsc_cond_reg}
holds with $Q(x) = O(x^2)$ as $|x| \rightarrow \infty$, and
$\sup_{\eta\in\widetilde\cH_0} P_\eta Q^2 < \infty$.
The proof of Theorem 6.2 in \cite{ghosal2001entropies} implies
that \eqref{eq:rate_condition_general} holds with
$\widetilde\cH_{n,1} = \widetilde\cH_0$,
$\widetilde\epsilon_n = n^{-1/2} (\log n)^{3/2}$ and any
$M_n \rightarrow\infty$.
Thus, \eqref{eq:rate_reg} hold with $\epsilon_n = n^{-1/2} (\log n)^2$
and $\cH_{n,1} = \cH_0$.

What remains to prove for the BvM assertion is asymptotic tightness
\cf\ \eqref{eq:score_process_reg}, which is implied if for every
$a\in\RR^p$ and sufficiently small $\epsilon > 0$, the stochastic process,
\begin{equation} \label{eq:score_process2}
	\bigg\{ (\theta,\eta) \mapsto \frac{a^T}{\sqrt{n}} \sum_{i=1}^n
	\left( \score_{\theta,\eta,i}- P_0 \score_{\theta,\eta,i} \right):
	\theta\in B_\epsilon, \eta\in\cH_0\bigg\},
\end{equation}
is asymptotically tight, where $B_\epsilon$ is the open ball of radius
$\epsilon$ centred on $\theta_0$. 
In Section \ref{subsubsec:7}, we prove
the asymptotic tightness of \eqref{eq:score_process2} using the
bracketing central limit theorem.
\end{proof}

The symmetrized Dirichlet process mixture prior considered in this
subsection is restricted, in that the mixing distribution $F$ is
supported on $[-M,M] \times [\sigma_1, \sigma_2]$. This restriction
plays only a technical role (to prove  \eqref{eq:Lipsc_cond_reg}
and \eqref{eq:score_process_reg}) and it is expected
that, with some additional effort, the results could be extended
to arbitrarily small $\sigma$'s and arbitraily large $M$.

\subsubsection{Random series prior}\label{sssec:rs}

Let $W$ be a random function on $[-1/2,1/2]$ defined as a series
$W(\cdot) = \sum_{j=1}^\infty j^{-\alpha} C_j b_j (\cdot)$, where
$b_1(t)=1, b_{2j}(t) = \cos(2\pi jt), b_{2j+1}(t) = \sin(2\pi jt)$ and
$C_j$'s are \iid\ random variables drawn from a density supported
on $[-M,M]$ that is continuous and bounded away from zero.
We shall impose smoothness through the requirement that $\alpha$
be greater than $3$, so that the series is well
defined as a continuous real-valued function on $[-1/2,1/2]$ with the
first and second derivatives that are bounded uniformly by a
constant.
Let $\cW$ be the set of all functions
$w:[-1/2,1/2]\rightarrow\RR$ of the form $w(\cdot) =\sum_j a_j b_j(\cdot)$
for some sequence $(a_1, a_2, \ldots)$ with $j^\alpha |a_j| \leq M$ for all $j$.
Let $\widetilde\cH_0$ denote the set of densities $p_w$, where $w\in\cW$ and,
\[
	p_w(x) = \frac{e^{w(x)}}{\int_{-1/2}^{1/2} e^{w(y)} dy},
\]
for every $x\in\DD = (-1/2,1/2)$. Let $\cH_0$ denote the associated
space of symmetrized $\bar{p}_w$. Let $\Pi_{\widetilde\cH}$ and
$\Pi_\cH$ be the laws of random densities $p_W$ and $\bar{p}_W$,
respectively.

\begin{corollary}\label{cor:ex-series}
Suppose that $\sup_{i\geq1}|Z_i| \leq L$ and 
$\liminf_n \rho_{\min}(\bZ_n) > 0$. If $\alpha > 3$, $\eta_0\in\cH_0$,
$v_{\eta_0} > 0$, and $\Pi_\Theta$ is compactly supported and thick
at $\theta_0$, then the random series prior $\Pi_\cH$ for $\eta$ leads
to a posterior for $\theta$ that satisfies the BvM assertion
\eqref{eq:bvmassertion} in the linear regression model.
\end{corollary}
\begin{proof}
We may assume that $\Theta$ is compact.
Let $W$ be the random function defined above, and let
$w_0(\cdot) = \sum_{j=1}^\infty j^{-\alpha} c_{0,j} b_j(\cdot)$ such that
$\eta_0(x) \propto e^{w_0(x)} + e^{w_0(-x)}$.
One verifies easily that the KL-divergence $K$, KL-variation $V$ and the
square Hellinger distance $h^2$, for densities $ p_w(\cdot) \propto
e^{w(\cdot)}$ are bounded by the square of the uniform norm of the difference 
between the exponents $w$.
Therefore by Lemma~\ref{lem:conv_rate},
conditions \eqref{eq:rate_reg} (with $\Theta_{n,1}=\Theta$ and
$\cH_{n,1} = \cH_0$) hold for some $(\epsilon_n)$ under the two
conditions: $\Pi_\cW \{\|W - w_0\|_\infty < \epsilon \} > 0$ and
$N(\epsilon, \cW, \|\cdot\|_\infty) < \infty$ for every
$\epsilon > 0$, where $\|\cdot\|_\infty$ is the uniform norm and
$\Pi_\cW$ is the law of $W$.
Since $\cW$ is totally bounded with respect to $\|\cdot\|_\infty$
by the Arzel{\`a}-Ascoli theorem, the condition 
$N(\epsilon, \cW, \|\cdot\|_\infty) < \infty$ is satisfied.
For given $\epsilon > 0$, there exists an integer $J$ such that
$M \cdot \sum_{j=J+1}^\infty j^{-\alpha} < \epsilon/4$.
Since each random variable $C_j$ has a positive and continuous
density at $c_{0,j}$ for $j \leq J$, we have $\Pi_\cW(A) > 0$ for
$A=\{\max_{j \leq J} |C_j - c_{0,j}| < 
\epsilon / (2\sum_{j=1}^\infty j^{-\alpha})\}$. Since
$\|W-w_0\|_\infty < \epsilon$ on $A$, we have
$\Pi_\cW \{\|W - w_0\|_\infty < \epsilon \} > 0$.

Note that \eqref{eq:Lipsc_cond_reg} is trivially satisfied with a
constant function $Q$. In Section~\ref{subsubsec:8}, we prove the asymptotic tightness of
\eqref{eq:score_process_reg}, which completes the proof.
\end{proof}

\section{Efficiency in the linear mixed effect model}
\label{sec:lm}

In this section, we consider the linear mixed effect model,
\[
  X_{ij} = \theta^T Z_{ij} + b_i^T W_{ij} + \epsilon_{ij},
  \quad\text{for $i=1,\ldots,n$ and $j=1, \ldots, m_i$},
\]
where the covariates $Z_{ij}\in\RR^p$ and $W_{ij}\in\RR^q$ are
non-random, the error $\epsilon_{ij}$'s form an \iid\ sequence drawn
from a distribution with density $f$ and the random effect
coefficients $b_i$ are \iid\ from a distribution $G$. 
The nuisance
parameter $\eta=(f,G)$ takes its values in the space
$\cH = \cF \times \cG$, where the first factor $\cF$ denotes the class
of continuously differentiable densities supported on $\DD=(-r,r)$
for some $r \in (0,\infty]$ with
$f(x)>0$ and $f(x)=f(-x)$ for all $x\in \DD$ and $\cG$ is the class of 
symmetric distributions supported on $[-M_b,M_b]^q$ for some $M_b > 0$.
The true value of the
nuisance is denoted by $\eta_0 = (f_0, G_0)$. We write
$X_i = (X_{i1}, \ldots, X_{im_i})^T$, and similarly,
$Z_i\in\RR^{p\times m_i}$ and $W_i\in\RR^{q\times m_i}$.
As in the linear regression model, we assume that,
\be \label{eq:lm_covariate}
  |Z_{ij}| \leq L\quad\text{and}\quad|W_{ij}| \leq L,
    \quad\text{for all $i$ and $j$}.
\ee
Define,
\[
  p_{\theta,\eta,i}(x) = \int \prod_{j=1}^{m_i} f(x_j -
  \theta^T Z_{ij} - b_i^T W_{ij}) dG(b_i),
\]
where $x=(x_1, \ldots, x_{m_i})^T \in\RR^{m_i}$. Quantities denoted
by $p_{\theta,\eta}^{(n)},
\ell_{\theta,\eta,i}, \score_{\theta,\eta,i}$ and $\score^{(n)}_{\theta,\eta}$
are defined and used in the same way as in Section~\ref{sec:regression}.
The design matrix $\bZ_n$ is defined by
$\bZ_n = n^{-1} \sum_{i=1}^n Z_i Z_i^T$. For technical reasons and
notational convenience, we assume that there exists an integer $m$
such that $m_i=m$ for all $i$, but proofs below can be extended to
general cases without much hamper. 

For $y=(y_1, \ldots, y_m)^T\in\RR^m$ and
$w=(w_1, \ldots, w_m) \in [-L,L]^{q\times m}$, define,
\[
  \psi_\eta(y|w) = \int \prod_{j=1}^{m} f(y_j - b^T w_j)\, dG(b),
\]
and $\ell_\eta(y|w) = \log\psi_\eta(y|w)$.
Let $s_\eta(y|w) = - \partial \ell_\eta(y|w) / \partial y \in \RR^m$.
Then it can be easily shown that 
$\score_{\theta,\eta,i}(x) = Z_i s_\eta\left(x-Z_i^T\theta| W_i\right) \in \RR^p$.
Furthermore, let
$\Psi^w_{\eta}(\cdot)$ denote the probability measure on $\RR^m$ with
density $y\mapsto\psi_\eta(y|w)$. 
The metric $h_n$ on $\Theta\times\cH$
is defined as in \eqref{eq:h_n-def}. With slight abuse of notation, we
also use $h_n$ as a metric on $\cH$ defined as
$h_n(\eta_1,\eta_2) = h_n((\theta_0,\eta_1),(\theta_0,\eta_2))$.
Let,
\[
	d^2_w(\eta_1,\eta_2) =
  \int |s_{\eta_1}(y|w) - s_{\eta_2}(y|w)|^2 d\Psi_{\eta_0}^w(y).
\]
Define $B_n(\epsilon)$ and $V_{n,\eta}$ as in \eqref{eq:B_n-def}
and \eqref{eq:V_n_eta-def}, respectively. 
It can be easily shown that,
\be\label{eq:V_n_def_lm}
  V_{n,\eta} 
  = \frac{1}{n}\sum_{i=1}^n Z_i v_{\eta}(W_i) Z_i^T,
\ee
where $v_\eta(w)$ is the $m\times m$ matrix defined as,
\[
  v_{\eta}(w) = \int s_{\eta}(y|w)\, s_{\eta_0}(y|w)^T
  \,d\Psi^w_{\eta_0}(y).
\]

To prove the BvM assertion in the linear mixed effect model, we need
a condition to ensure that $\sup_{i\geq 1} h(\psi_{\eta_n}(\cdot|W_i),
\psi_{\eta_0}(\cdot|W_i)) \rightarrow $ as $h_n(\eta_n,\eta_0)\rightarrow 0$.
For this purpose, we define $N_{n,\epsilon}(u)$ to be the number of
$W_{ij}$'s with $|W_{ij} - u| < \epsilon$, and assume that,
for every (fixed) $\epsilon > 0$ and $u \in \RR^q$,
\begin{equation} \label{eq:design}
	N_{n,\epsilon}(u)=0 \;\; \textrm{for all $n$,} \quad\text{or}\quad \liminf_n n^{-1} N_{n,\epsilon}(u) > 0.
\end{equation}
Condition \eqref{eq:design}
is easily satisfied, for example when $W_{ij}$'s are \iid\ realization
from any distribution.

\begin{theorem}
\label{th:main-lm}
Suppose that
$\liminf_n \rho_{\rm min}(\bZ_n)>0$,
$\rho_{\rm min}(v_{\eta_{0}}(w)) > 0$ for every $w$, $G_0$ is thick at 0,
$\Pi_\Theta$ is thick at $\theta_0$, and $w\mapsto v_{\eta_0}(w)$ is continuous.
Also suppose that there exist a large integer $N$, a sequence $(\epsilon_n)$,
with $\epsilon_n \downarrow 0$ and $n\epsilon_n^2\rightarrow\infty$, 
and sequences of partitions
$\Theta=\Theta_{n,1} \cup \Theta_{n,2}$, $\cH=\cH_{n,1}\cup \cH_{n,2}$
such that $\eta_0\in \cH_{n,1}$ and \eqref{eq:rate_reg} holds for all
$n \geq N$. For some $\overline M_n\uparrow\infty$, with
$\epsilon_n \overline M_n \rightarrow 0$, let
$\cH_n = \{\eta\in\cH_{n,1}: h_n(\eta,\eta_0) < \overline M_n\epsilon_n\}$.
Assume that there exists a continuous function $Q$ such that
$\sup_w \int Q^3(x,w) \psi_{\eta_0}(x|w) d\mu(x) < \infty$, and,
\begin{equation}
  \label{eq:Lipsc_cond_lm}
  \sup_{\eta\in\cH^N}\frac{|\ell_{\eta}(x+y|w)- \ell_{\eta}(x|w)|}{|y|}
  \vee \frac{|s_{\eta}(x+y|w)- s_{\eta}(x|w)|}{|y|}  \leq  Q(x,w),
\end{equation}
for all $x, w$ and small enough $|y|$, where $\cH^N  = \cup_{n=N}^\infty \cH_n$.
Also assume that the class of $\RR^2$-valued functions,
\bea
  &&\bigg\{w \mapsto \Big(\,d_w(\eta_1,\eta_2),\,
  h(\psi_{\eta_1}(\cdot|w), \psi_{\eta_2}(\cdot|w))\, \Big):
   \eta_1, \eta_2\in\cH^N \bigg\}, \label{eq:w_class}
\eea
is equicontinuous, and for sufficiently small
$\epsilon_0 >0$ the stochastic process,
\begin{equation}
\label{eq:score_process_lm}
  \bigg\{\frac{1}{\sqrt{n}} \Big(\score_{\theta,\eta}^{(n)}
    - P_0^{(n)} \score_{\theta,\eta}^{(n)}\Big):
    |\theta-\theta_0| <\epsilon_0, \eta\in \cH^N \bigg\},
\end{equation}
is asymptotically tight.
Then, the BvM assertion
\eqref{eq:bvmassertion} holds for the linear mixed effect model,
provided that \eqref{eq:lm_covariate} and \eqref{eq:design} hold.
\end{theorem}

The proof of Theorem \ref{th:main-lm} is quite similar to that of
Theorem \ref{th:main-linear} except for some technical details.
Below we follow the same line to the proof of Theorem \ref{th:main-linear}.

\subsection{Proof of the misspecified LAN property}

Let \eqref{eq:AB_n} define $A_n(h,\eta)$ and $B_n(h,\eta)$ again and let $K$
be a compact subset of $\RR^p$. Then it suffices to prove that
$A_n(h,\eta)$ and $B_n(h,\eta)$ converge in $P_0^{(n)}$-probability to
zero uniformly over $h\in K$ and $\eta\in\cH^N$.
Note that the inequality \eqref{eq:tight_process} still holds. Since,
\[
  \begin{split}
  {\rm Var} \bigg(\frac{1}{\sqrt{n}}
  &\Big(\score_{\theta_n(h),\eta}^{(n)}
      -\score_{\theta_0,\eta}^{(n)}\Big)^o \bigg)
    = \frac{1}{n}\sum_{i=1}^n P_0 | \score_{\theta_n(h),\eta,i}
      - \score_{\theta_0,\eta,i}|^2\\
  &= \frac{1}{n}\sum_{i=1}^n P_0 \Big| Z_i
    \Big(s_\eta(X_i - Z_i^T \theta_n(h)| W_i)
      - s_\eta(X_i - Z_i^T \theta_0| W_i) \Big) \Big|^2\\
  &\leq \frac{1}{n}\sum_{i=1}^n \|Z_i\|^4 \cdot
    |\theta_n(h) - \theta_0|^2 \cdot P_0 Q(X_i, W_i)^2 = o(1),
  \end{split}
\]
$\sup_{h\in K} \sup_{\eta\in\cH^N} |A_n(h,\eta)| = o_{P_0}(1)$ by
asymptotic tightness of \eqref{eq:score_process_lm}.

For $B_n(h,\eta)$, we prove in Section \ref{subsubsec:4-1} that,
\begin{equation}\label{eq:unif_quad_lm}
  \sup_{\eta\in\cH^N} \left| \frac{1}{n}P_0^{(n)}
  \Big( \ell^{(n)}_{\theta,\eta} - \ell^{(n)}_{\theta_0,\eta}\Big)
  + \frac{1}{2} (\theta-\theta_0)^T V_{n,\eta} (\theta-\theta_0) \right|
  = o(|\theta-\theta_0|^2),
\end{equation}
as $\theta \rightarrow \theta_0$. Consequently, the supremum of
$B_n(h,\eta)$ over $h \in K$ and $\eta\in \cH^N$ converges to 0.
\qed

\subsection{Proof of Condition A}

It is shown in Section \ref{subsubsec:4-2} that,
\be \label{eq:d_w_consistency}
  \lim_{n\rightarrow\infty} \sup_{i \geq 1} \sup_{\eta\in\cH_n}
  d_{W_i}(\eta,\eta_0) = 0.
\ee
Note that for any $a \in \RR^m$ with $|a|=1$,
\be
  \label{eq:v_eta_w-consistency}
  \begin{split}
  a^T&\Big(v_\eta(w) - v_{\eta_0}(w)\Big)a\\
  &= \int a^T \Big(s_\eta(x|w) - s_{\eta_0}(x|w)\Big)
    s_{\eta_0}(x|w)^T a\; d\Psi^w_{\eta_0}(x)\\
  &\leq C \int \Big| s_\eta(x|w) - s_{\eta_0}(x|w)\Big|^2 d\Psi^w_{\eta_0}(x)
    = C d_w^2(\eta, \eta_0),
\end{split}
\ee
for some constant $C > 0$ by the Cauchy-Schwartz inequality and
\eqref{eq:Lipsc_cond_lm}. Thus, 
\[
  \lim_{n\rightarrow\infty} \sup_{i\geq 1} \sup_{\eta\in\cH_n}
  \|v_\eta(W_i) - v_{\eta_0}(W_i)\| = 0.
\]
Since,
\bean
  \sup_{\eta\in\cH_n} \|V_{n,\eta} - V_{n,\eta_0}\| 
    = \sup_{\eta\in\cH_n} \bigg\| \frac{1}{n} \sum_{i=1}^n Z_i 
    \{v_\eta(W_i) - v_{\eta_0}(W_i)\} Z_i^T \bigg\|= o(1),
\eean
which completes the proof of \eqref{eq:V_conti_condition}.

Let $a \in \RR^p$ be a fixed non-zero vector.
Then for any sequence $\eta_n \in \cH_n$,
\[
  {\rm Var} \bigg(\frac{a^T}{\sqrt{n}} (\score^{(n)}_{\theta_0,\eta_n}
    - \score^{(n)}_{\theta_0,\eta_0})\bigg)
  = \frac{1}{n} \sum_{i=1}^n a^T Z_i u_{\eta_n}(W_i) Z_i^T a,
\]
where,
\[
  u_\eta(w) = \int \Big(s_\eta(x|w) - s_{\eta_0}(x|w)\Big)
  \Big(s_\eta(x|w) - s_{\eta_0}(x|w) \Big)^T d\Psi_{\eta_0}^{w}(x).
\]
Since $|b^T u_\eta(w) b| \leq d_w(\eta,\eta_0)$ for every
$\eta\in\cH^N$ and $b\in\RR^m$ with $|b| = 1$, 
we have $\sup_{i\geq 1} \|u_{\eta_n}(W_i)\| = o(1)$
by \eqref{eq:d_w_consistency}, and so,
\[
  \frac{a^T}{\sqrt{n}} (\score^{(n)}_{\theta_0,\eta_n}
    - \score^{(n)}_{\theta_0,\eta_0}) = o_{P_0}(1).
\]
For given $\epsilon, \delta > 0$, by asymptotic tightness of
\eqref{eq:score_process_lm} and Theorem~1.5.6 of \cite{van1996weak},
there is a partition $\cH^N = \cup_{j=1}^J \cH^{(j)}$ such that,
\[
  P_0\bigg(\max_{1\leq j\leq J}\sup_{\eta_1,\eta_2\in\cH^{(j)}} 
    \bigg|\frac{a^T}{\sqrt{n}} (\score^{(n)}_{\theta_0,\eta_1} 
      - \score^{(n)}_{\theta_0,\eta_2}) \bigg| > \epsilon \bigg) < \delta,
\]
for large enough $n$.
We can choose sequences $(\eta_n^{(j)})$ for $j=1, \ldots, J$ such
that $\eta_n^{(j)} \in \cH_n$ and for every $n\geq N$ and for a given
$\eta\in\cH_n$ there exists at least one $j$ such that $\eta$ and
$\eta_n^{(j)}$ are contained in the same partition. Since,
\[
  \max_{1\leq j \leq J}\bigg|\frac{a^T}{\sqrt{n}}
  (\score^{(n)}_{\theta_0,\eta_n^{(j)}} 
  - \score^{(n)}_{\theta_0,\eta_0})\bigg| = o_{P_0}(1),
\]
we have,
\[
  P_0\bigg(\sup_{\eta\in\cH_n} \bigg|\frac{a^T}{\sqrt{n}}
    (\score^{(n)}_{\theta_0,\eta} - \score^{(n)}_{\theta_0,\eta_0})
    \bigg| > 2\epsilon \bigg) < 2\delta,
\]
for large enough $n$. Since $a$ is an arbitrary vector,
\eqref{eq:score_conti_condition} is proved.

Since $\rho_{\rm min} (v_{\eta_{0}}(w)) > 0$ for every $w$ and the
map $w\mapsto v_{\eta_0}(w)$ is continuous, we have that
$\inf_{w\in [-L,L]^{q\times m}} \rho_{\rm min} (v_{\eta_{0}}(w)) > 0$.
In addition, since each component of the matrix $v_{\eta_0}(w)$ is
bounded uniformly in $w$ by the integrability condition 
\eqref{eq:Lipsc_cond_lm}, we have $\sup_{w\in [-L,L]^{q\times m}}
\rho_{\rm max} (v_{\eta_{0}}(w)) < \infty$. Finally, since,
\[
  0 < \liminf_n \rho_{\min}(\bZ_n) \leq \limsup_n \rho_{\max}(\bZ_n) < \infty,
\]
\eqref{eq:V_positive} is satisfied by \eqref{eq:V_n_def_lm}.

\subsection{Proof of Condition B}

We shall have need for the following lemma, the proof of which is in
Section~\ref{subsubsec:4-3}.

\begin{lemma}
\label{lem:lm_consistency1} Under the conditions in Theorem
\ref{th:main-lm}, there exists $K>0$ such that for every
sufficiently small $\epsilon > 0$ and $\eta\in\cH^N$,
$h_n\big((\theta,\eta),(\theta_0,\eta_0)\big) < \epsilon$ implies 
$|\theta-\theta_0| < K\epsilon$ and $h_n(\eta,\eta_0) < K\epsilon$. 
\end{lemma}

Posterior consistency of the parameter $(\theta,\eta)$ with
respect to the metric $h_n$ is guaranteed by Theorem~4 of
\cite{ghosal2007convergence}. Thus, Lemma \ref{lem:lm_consistency1}
implies \eqref{eq:eta_consistency}. The proof of
\eqref{eq:sqrtn_general} for the linear mixed effect model
is very similar to the analogous proof in linear regression model,
as in Section~3.

\subsection{Examples}

Let $\widetilde \cF$ (resp. $\widetilde \cG$) be the set of every $f$
(resp. $G$) whose symmetrization $\bar f$ (resp. $\overline G$)
belongs to $\cF$ (resp. $\cG$), where $\overline G=(G + G^-)/2$ with
$G^-(A) = G(-A)$ for every measurable set $A$. For the prior of $\eta$, we
consider a product
measure $\Pi_\cF\times \Pi_\cG$, where $\Pi_\cF$ and $\Pi_\cG$ are the
symmetrized versions of probability measures $\Pi_{\widetilde\cF}$
and $\Pi_{\widetilde\cG}$ on $\widetilde \cF$  and $\widetilde \cG$,
respectively.
The following lemma plays a role in the proof of
Corollary~\ref{cor:ex-dpm-lm} (its proof is given in Section
\ref{subsubsec:4-4}).
Denote the L\'{e}vy-Prohorov metric between
two probability measures $P_1$, $P_2$ is denoted by $d_W(P_1,P_2)$.

\begin{lemma} \label{lem:lm-consistency2}
Let $\cH_0 = \cF_0\times\cG_0 \subset \cH$ for some $\cF_0 \subset
\cF$ and $\cG_0 \subset \cG$ with $f_0\in\cF_0$ and $G_0\in\cG_0$.
Assume that there exist a continuous function $Q_0$ and small enough
$\delta_0 > 0$ such that,
\be\label{eq:Q_0-integrability}
  \int \sup_w \sup_{\eta\in\cH_0} Q_0(x,w)^2 \psi_{\eta}(x|w) d\mu(x)
  < \infty,
\ee
and,
\begin{equation} \label{eq:Q_0-lip}
  \sup_{\eta\in\cH_0}\frac{|\ell_{\eta}(x+y|w)- \ell_{\eta}(x|w)|}{|y|}
  \vee \bigg| \frac{\psi_{\eta_0}(x|w)}{\psi_\eta(x|w)} \bigg|^{\delta_0}
  \leq  Q_0(x,w),
\end{equation}
for all $x, w$ and small enough $|y|$.
Also assume that $\cF_0$ is uniformly tight and,
\be\label{eq:cF_0-bound}
  \sup_{f\in\cF_0}\sup_x f(x) \vee |\dot f(x)|< \infty,
\ee
where $\dot f$ is the derivative of $f$.
Then, on $\Theta\times\cH_0$,
\be\label{eq:h_n-consistency-lm}
  \sup_{n\geq 1} h_n\big((\theta_1, \eta_1), (\theta_2, \eta_2)\big)
    \rightarrow 0,
\ee
as $|\theta_1-\theta_2| \vee h(f_1, f_2) \vee d_W(G_1, G_2) \rightarrow 0$, and,
\be\label{eq:KL-consistency-lm}
  \sup_{n \geq 1}\frac{1}{n}\sum_{i=1}^n 
    K(p_{\theta_0,\eta_0,i}, p_{\theta,\eta,i}) 
    \vee V(p_{\theta_0,\eta_0,i}, p_{\theta,\eta,i}) \rightarrow 0,
\ee
as $|\theta-\theta_0| \vee h(f, f_0) \vee d_W(G, G_0) \rightarrow 0$.
\end{lemma}

\subsubsection{Symmetric Dirichlet mixtures of normal distributions}
\label{sssec:ex-dpm-lm}

Let $\Pi_\cF$ denote the prior for the symmetric Dirichlet mixtures of
normal distributions defined in Section \ref{sssec:dpm} and let
$\cF_0$ be the support of $\Pi_\cF$ in Hellinger metric.
Let $\cG_0$ be the support of a prior $\Pi_\cG$ on $\cG$ in the
weak topology, and let $\cH_0 = \cF_0 \times\cG_0$. The following
corollary proves the BvM theorem for $\theta$.

\begin{corollary}\label{cor:ex-dpm-lm}
Assume that $\liminf_n \rho_{\rm min}(\bZ_n)>0$.
With the prior $\Pi_\cH$ described above, the BvM theorem holds
for the linear mixed regression model provided that $\eta_0 \in
\cH_0$, $G_0$ is thick at 0, and $\Pi_\Theta$ is compactly supported
and thick at $\theta_0$, provided \eqref{eq:lm_covariate} and
\eqref{eq:design} hold. 
\end{corollary}
\begin{proof}
We may assume that $\Theta$ is compact, and let $\Theta_{n,1} = \Theta$ and
$\cH_{n,1} = \cH_0$ for all $n \geq 1$.
It is easy to show that $\rho_{\min}(v_{\eta_0}(w)) > 0$ for every $w$
and $w \mapsto v_{\eta_0}(w)$ is continuous.
We prove in Section \ref{subsubsec:4-4.1} that
\be
  \label{eq:density_bound-lm}
  \begin{split}
  C_1 \exp(-C_2 |x|^2) &\leq \inf_w \inf_{\eta\in\cH_0}\psi_\eta(x|w)\\
    &\leq \sup_w \sup_{\eta\in\cH_0} \psi_\eta(x|w) \leq C_3 \exp(-C_4 |x|^2)
  \end{split}
\ee
for some constants $C_1, C_2, C_3, C_4 > 0$ and large enough $|x|$.
Also, the first and second order partial derivative of
$x\mapsto \ell_\eta(x|w)$ are of order $O(|x|^2)$ as
$|x| \rightarrow \infty$ for every $\eta \in \cH_0$ and $w$,
so, with $Q(x,w) = C_5 (1 + |x|^2)$ for some $C_5 > 0$, we have,
\be\label{eq:Q-dpm-lm}
  \sup_{\eta\in\cH_0}\frac{|\ell_{\eta}(x+y|w)- \ell_{\eta}(x|w)|}{|y|}
  \vee \frac{|s_{\eta}(x+y|w)- s_{\eta}(x|w)|}{|y|}  \leq  Q(x,w),
\ee
for every $x, w$ and small enough $|y|$, and,
\be\label{eq:Q-dp-lm}
  \int \sup_w \sup_{\eta\in\cH_0} Q^3(x,w) \psi_{\eta}(x|w) d\mu(x) < \infty.
\ee

We next prove \eqref{eq:rate_reg} with the help of Lemma
\ref{lem:lm-consistency2}. Since $\Pi_\Theta(\Theta_{n,1}) = 
\Pi_\cH(\cH_{n,1}) = 1$, the third inequality of \eqref{eq:rate_reg}
holds trivially. By \eqref{eq:density_bound-lm},
\bean
  \int \sup_{\eta_1, \eta_2\in\cH_0} \bigg|
    \frac{\psi_{\eta_0}(x|w)}{\psi_\eta(x|w)} \bigg|^{2\delta_0} 
  \psi_{\eta_2}(x|w) d\mu(x) < \infty,
\eean
for sufficiently small $\delta_0 > 0$, so combining with
\eqref{eq:Q-dp-lm}, \eqref{eq:Q_0-integrability} and
\eqref{eq:Q_0-lip} hold for some $Q_0$. Uniform tightness of $\cF_0$
and \eqref{eq:cF_0-bound} is easily satisfied, so the conclusion of
Lemma \ref{lem:lm-consistency2} holds. By
\eqref{eq:h_n-consistency-lm}, the first inequality of
\eqref{eq:rate_reg} holds for some rate sequence $\epsilon_{n,1}$
because $\Theta\times\cF_0\times\cG_0$ is totally bounded with respect
to the product metric $|\cdot| \times h \times d_W$.
Also, by \eqref{eq:KL-consistency-lm}, the second inequality of
\eqref{eq:rate_reg} holds for some $\epsilon_{n,2}$ because every
$|\cdot| \times h \times d_W$ neighborhoods of $(\theta_0, f_0, G_0)$
has positive prior mass. Thus, \eqref{eq:rate_reg} holds with
$\epsilon_n = \max\{\epsilon_{n,1}, \epsilon_{n,2}\}$.

To complete the proof, equicontinuity of \eqref{eq:w_class} is
proved in Section~\ref{subsubsec:4-5} and asymptotic tightness
of \eqref{eq:score_process_lm} in
Section~\ref{subsubsec:4-6}.
\end{proof}
It should be noted that the only condition
for $\Pi_\cG$ is that $G_0\in \cG_0$. Thus, we can consider both
parametric and nonparametric priors for $G$. For example, the multivariate
normal distribution truncated on $[-M_b,M_b]^q$ or the symmetrized
${\rm DP}(\alpha, H_G)$ prior with a distribution $H_G$ on
$[-M_b,M_b]^q$ can be used for $\Pi_\cG$.

\subsubsection{Random series prior}

Let $\Pi_\cF$ be the random series prior defined in
Section~\ref{sssec:rs} and let
$\cF_0$ be the support of $\Pi_\cF$.
Since the distributions in $\cF_0$ have compact supports, the
distributions in $\cG_0$, the support of $\Pi_\cG$, should 
have the same support for  \eqref{eq:Lipsc_cond_lm} to hold. 
Hence, we only  consider truncated normal distributions truncated
on $[-M_b,M_b]^q$ with positive definite covariance matrixes.
That is, $\cG_0=\{N_{M_b}(0,\Sigma): 0<\rho_1\le \rho_{\min}(\Sigma)
\le \rho_{\max}(\Sigma)\le \rho_2<\infty\}$ for some
constants $\rho_1$ and $\rho_2$, where $N_{M_b}(0, \Sigma)$ denotes
the truncated normal distribution. Let $\Pi_\cH=\Pi_\cF\times \Pi_\cG$.

 \begin{corollary}\label{cor:ex-rs-lm}
Assume that $\liminf_n \rho_{\rm min}(\bZ_n)>0$ and
$\rho_{\min}( v_{\eta_0}(w)) > 0$ for every $w$.
With the prior $\Pi_\cH$ described above, the BvM theorem holds
for the linear mixed regression model provided that $\eta_0 \in
\cH_0$, and $\Pi_\Theta$ is compactly supported and thick at
$\theta_0$ provided \eqref{eq:lm_covariate} and \eqref{eq:design} hold.
\end{corollary}
\begin{proof}
Replacing $Q$ and $Q_0$ as constant functions, the proof is almost
identical to that of Corollary~\ref{cor:ex-dpm-lm}, except for the
proof of asymptotic tightness of \eqref{eq:score_process_lm}, which
is proved in Section~\ref{subsubsec:4-7}.
\end{proof}

\section{Numerical study}\label{sec:simulation}

In this section, we provide simulation results to illustrate
semi-parametric efficacy of the Bayes estimator in the linear
mixed effect model. We specialize the model introduced in
section~\ref{sec:lm} slightly: we only consider the random intercept
model,
\begin{equation}
  \label{eq:lme}
  X_{ij} = \theta^T Z_{ij} + b_i + \epsilon_{ij},
\end{equation}
where the $b_i$'s are univariate random effects following a
normal distribution with mean 0 and variance $\sigma_b^2$.
\begin{figure}
\caption{Density plots of error distribution in E4 (left) and E5 (right).}
\label{fig:density}
\bc \includegraphics[width=100 mm, height=45 mm]
  {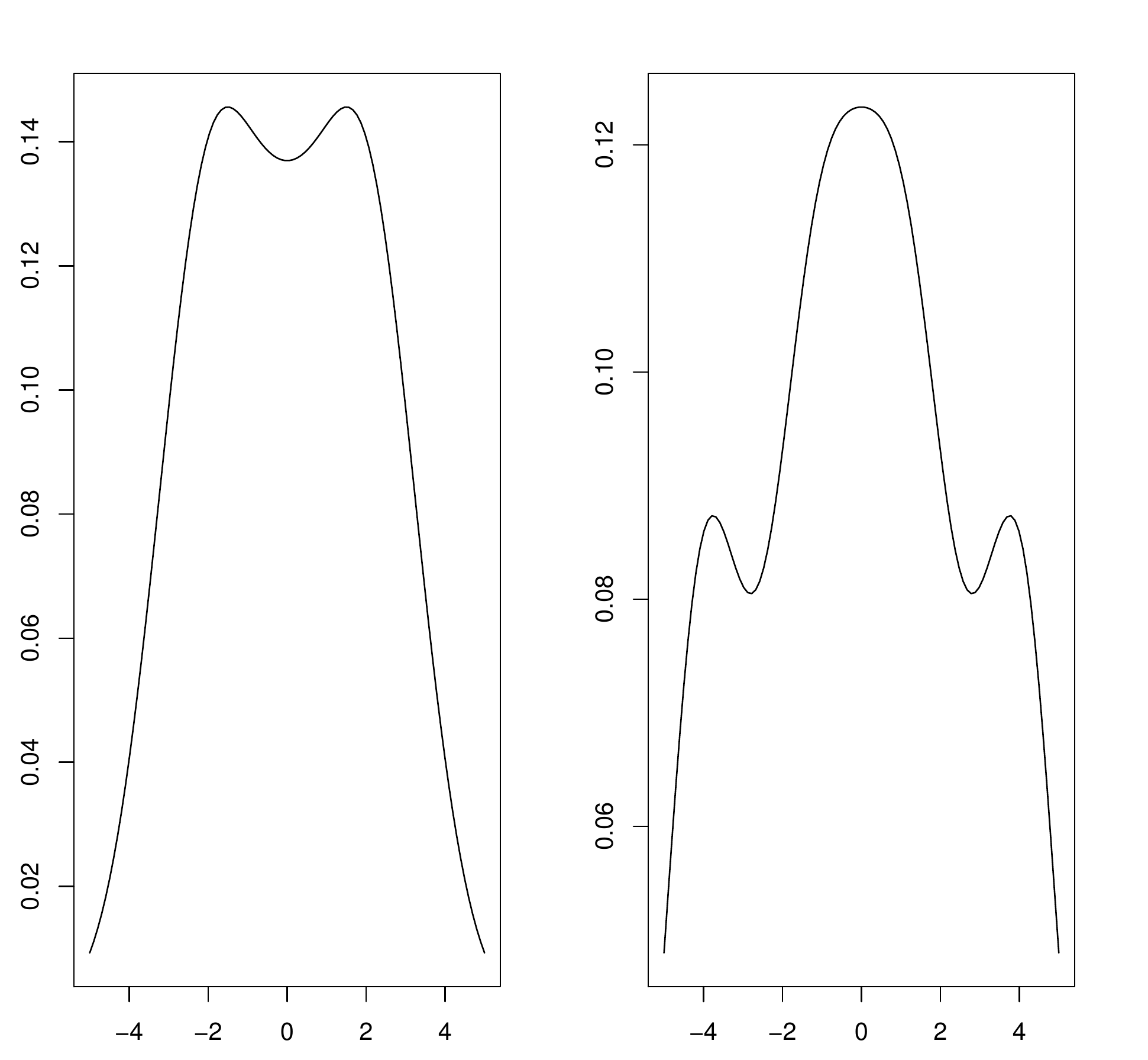} \ec
\end{figure} 
In simulations, a dataset is generated from model \eqref{eq:lme}
with various error distributions. Then, the regression parameters
$\theta$ are estimated using various methods including both
frequentist and Bayesian approaches for comparison. This
procedure is repeated $N$ times and the performance of
estimation methods is evaluated by mean squared error, 
$N^{-1} \sum_{k=1}^N |\hat\theta^{(k)}_n - \theta_0|^2$,
where $\hat\theta^{(k)}_n$ is the estimate in the $k$th simulation.
We compare the performance of 3 estimators under 5 error
distributions. In all simulations we let
$Z_{ij} = (Z_{ij1}, Z_{ij2})^T$, where the $Z_{ijk}$'s
are generated \iid\ from the Bernoulli distribution
with success probability 1/2. The true parameters $\theta_0$
and $\sigma^2_{0b}$ are set to be $(-1,1)^T$ and 1, respectively.
For the error distribution, we consider the standard normal
distribution (E1), the Student $t$-distributions with 2 degree of
freedom (E2), the uniform(-3,3) distribution (E3), and two mixtures of
normal distribution (E4 and E5). For the mixtures we take,
\[
  p(x) = \sum_{k=1}^K \pi_k\Big(\phi_1 (x-\mu_k) + \phi_1 (x+\mu_k)\Big),
\]
with $K=4$, 
\bean
	(\mu_1,\mu_2,\mu_3,\mu_4)=(0,1.5,2.5,3.5); \quad
	(\pi_1,\pi_2,\pi_3,\pi_4)=(0.1,0.2,0.15,0.05),
\eean
for E4, and $K=4$, 
\bean
	(\mu_1,\mu_2,\mu_3,\mu_4)=(0,1,2,4); \quad
	(\pi_1,\pi_2,\pi_3,\pi_4)=(0.05,0.15,0.1,0.2),
\eean
for E5.
These two densities (see Figure~\ref{fig:density}) have two
and three modes, respectively.

For the estimators of $\theta$, we consider one frequentist estimator
(F) (the maximum likelihood estimator under the assumption of a
normal error and normal random effect, which is equal to Henderson's
{\it best linear unbiased estimator} \cite{henderson1975best}), and
two Bayesian estimators
(B1 and B2). For the two Bayes estimators, we consider two
different priors for the distribution of $\eta$: the normal
distributions with mean 0 and variance $\sigma_\epsilon^2$ for
$f$ and normal distribution with mean 0 and variance $\sigma_b^2$
for $G$ (B1), and a symmetrized Dirichlet process mixture for $f$
and normal distribution with mean 0 and variance $\sigma_b^2$
for $G$ (B2). Independent inverse Gamma distributions are
used for the priors of $\sigma_\epsilon^2$
and $\sigma_b^2$, and independent diffuse
normal distributions are used for the prior of $\theta$.

\begin{table*}
\caption{Mean squared error (and relative efficiency with respect
to B2) of each methods F, B1 and B2 among $N=300$ repetitions for
each experiment E1--E5.}
\label{tab:sim_res}
\bc
\begin{tabular}{c|rrr}
	\hline
	\multicolumn{2}{r}{F}   & B1    & B2   \\
	\hline
	E1	& 0.03		& 0.03		& 0.03 \\
		& (0.98)	& (0.98)	& (1.00) \\
	E2	& 0.27		& 0.26		& 0.09 \\
		& (3.06)	& (2.99)	& (1.00) \\
	E3	& 0.07		& 0.07		& 0.05 \\
		& (1.40)	& (1.39)	& (1.00) \\
	E4	& 0.13		& 0.12		& 0.11 \\
		& (1.18)	& (1.16)	& (1.00) \\
	E5	& 0.19		& 0.19		& 0.17 \\
		& (1.13)	& (1.12)	& (1.00) \\
  \hline
\end{tabular}
\ec
\end{table*}

For each error distribution, $N=300$ datasets with $n=20$ and
$m_i=5$ for all $i$, are generated. The mean squared errors and
relative efficiencies (with respect to B2) of the three estimators
are summarized in Table~\ref{tab:sim_res}. B2 dominates the other
two estimators when the error distribution is other than the normal.
In particular, the losses of efficiency  for F and B1 compared
to B2 are relatively large when the error distribution
has a heavier tail than the normal distribution (\eg\ E2).

\appendix

\section{Appendix}\label{sec:appendix}

\subsection{Proofs for Section \ref{sec:regression}}\label{ssec:1}

\subsubsection{Proof of \eqref{eq:unif_quad_reg}}\label{subsubsec:1}
Since,
\bean
  & & \sup_{\eta\in\cH^N} \left| \frac{1}{n}P_0^{(n)}
  \Big( \ell^{(n)}_{\theta,\eta} - \ell^{(n)}_{\theta_0,\eta}\Big)
  + \frac{1}{2} (\theta-\theta_0)^T V_{n,\eta} (\theta-\theta_0) \right|\\
 &\le& \sup_{i\geq 1} \sup_{\eta\in\cH^N} \left| P_{\eta_0}
  \log \frac{\eta\big(X-(\theta-\theta_0)^T Z_i\big)}{\eta(X)}
  + \frac{1}{2} v_\eta (\theta-\theta_0)^T Z_i Z_i^T (\theta-\theta_0)\right|,
\eean
where $X\sim P_{\eta_0}$, it suffices to show that,
\begin{equation}
  \label{eq:uql}
   \sup_{\eta\in \cH^N}\Bigl| P_{\eta_0}\log\frac{\eta(X-y)}
    {\eta(X)}+\frac{y^2}{2}v_{\eta} \Bigr|=o(y^2),
\end{equation}
as $y\rightarrow0$.

We consider only the case $y > 0$; the case $y<0$ is treated
similarly. For $\eta\in \cH^N$, we have,
\be
\begin{split}\label{eq:qil_pf}
  & \int \log \frac{\eta(x-y)}{\eta(x)} \eta_0(x) dx
    \\
  & = \int_{-\infty}^0 \log \frac{\eta(x-y/2)}{\eta(x+y/2)} \eta_0(x+y/2) dx
    + \int_0^\infty \log \frac{\eta(x-y/2)}{\eta(x+y/2)} \eta_0(x+y/2) dx
    \\
  & = \int_0^\infty \log \frac{\eta(-x-y/2)}{\eta(-x+y/2)} \eta_0(-x+y/2) dx
    + \int_0^\infty \log \frac{\eta(x-y/2)}{\eta(x+y/2)} \eta_0(x+y/2) dx
    \\
  & = \int_0^\infty \log \frac{\eta(x+y/2)}{\eta(x-y/2)} \eta_0(x-y/2) dx
    + \int_0^\infty \log \frac{\eta(x-y/2)}{\eta(x+y/2)} \eta_0(x+y/2) dx
    \\
  & = -\int_0^\infty \left[ \ell_\eta\Big(x - \frac{y}{2}\Big) 
    - \ell_\eta \Big(x + \frac{y}{2}\Big) \right] \cdot
    \left[ {\eta_0}\Big(x-\frac{y}{2} \Big) - 
    {\eta_0}\Big(x+\frac{y}{2} \Big) \right] dx
    \\
  & = -\int_{-y/2}^\infty \Big[\ell_\eta(x+y) - \ell_\eta(x)\Big] \cdot
    \Big[{\eta_0}(x+y) - {\eta_0}(x)\Big] dx
    \\
  & = -\int_0^\infty \Big[\ell_\eta(x+y) - \ell_\eta(x)\Big] \cdot
    \Big[{\eta_0}(x+y) - {\eta_0}(x)\Big] dx + R(y,\eta),
\end{split}
\ee
where the third equality holds by the symmetry of $\eta$ and $\eta_0$,
and,
\[
  R_n(y,\eta) = -\int_{-y/2}^0
  \Big[\ell_\eta(x+y) - \ell_\eta(x)\Big] \cdot
  \Big[{\eta_0}(x+y) - {\eta_0}(x)\Big] dx.
\]
Note that $\sup_{\eta\in\cH^N} |R(y,\eta)| = o(y^2)$ as
$y \rightarrow 0$ because (\ref{eq:Lipsc_cond_reg}) implies,
\bean
  |R(y,\eta)| &=& 
    \bigg|\int_{-y/2}^0 \Big[\ell_\eta(x+y) - \ell_\eta(x)\Big] \cdot
    \Big[{\eta_0}(x+y) - {\eta_0}(x)\Big] dx\bigg|
    \\
  &=& y \cdot \bigg|\int_{-y/2}^0 \int_0^1
    \Big[\ell_\eta(x+y) - \ell_\eta(x)\Big] \cdot
    \dot\eta_0(x+ty) dt\; dx\bigg|
    \\
  &\leq& 2y^2\int_0^1 \int_{-y/2}^0 Q(x+ty)\cdot
    |s_{\eta_0}(x+ty)| \cdot \eta_0(x+ty) dx\;dt
    \\
  &\leq& 2y^2\int_0^1 \int_{-y/2}^0 Q^2(x+ty) \cdot \eta_0(x+ty) dx\;dt
    \lesssim y^3
\eean
for small enough $y$ by the continuity of $Q$ and $\eta_0$,
where $\dot\eta(x) = \partial\eta(x)/\partial x$.
Finally, a Taylor expansion and Fubini's theorem imply that
the last integral of \eqref{eq:qil_pf} is equal to,
\be\label{eq:y_s_eta}
  y^2 \int_0^1\int_0^1\int_0^\infty s_\eta(x+ty) \dot\eta_0 (x+sy)\;
  dx\, dt\, ds.
\ee
Since,
\be\label{eq:y_vy}
\frac{y^2}{2} v_\eta = -y^2\int_0^\infty s_\eta(x)\;{\dot\eta_0}(x) dx,
\ee
the sum of \eqref{eq:y_s_eta} and \eqref{eq:y_vy} is bounded by,
\bean
  && y^2\bigg|\int_0^1\int_0^1 \int_0^\infty 
    s_\eta(x+ty) \dot\eta_0 (x+sy) - s_\eta(x) \dot\eta_0 (x) 
    \; dx\, dt\, ds\bigg|\\
  &&\leq y^2 \int_0^1\int_0^1 \int_0^\infty 
    \Big| \Big\{s_\eta(x+ty) - s_\eta(x+sy)\Big\}\dot\eta_0(x+sy)
    \Big|\; dx\, dt\, ds\\
  && ~~~ + \;y^2 \bigg|\int_0^1 \int_0^\infty \Big[
    s_\eta(x+sy) \dot\eta_0 (x+sy) - s_\eta(x)\dot\eta_0 (x) \Big]
    \; dx\, ds\bigg|
    \\
  &&\leq y^3 \int Q(x) |\dot\eta_0 (x)| dx
    + y^2 \sup_{s\in[0,1]} \bigg|\int_0^{sy} s_\eta(x) \dot\eta_0 (x) dx\bigg|
    \\
  &&\leq y^3 P_{\eta_0}Q^2 + y^2 \int_0^y Q^2(x) \eta_0(x) dx
    = O(y^3),
\eean
as $y \rightarrow 0$.
\qed

\subsubsection{Proof of \eqref{eq:d2_domination}}\label{subsubsec:2}

For a sequence $(\eta_n)$ such that $\eta_n \in \cH_n$ and
$\sup_{\eta\in\cH_n} d_2(\eta,\eta_0) < d_2(\eta_n,\eta_0) + n^{-1}$,
it suffices to show that $d_2(\eta_n,\eta_0) \rightarrow 0$.
By the definition of $\cH_n$, we have $h(\eta_n, \eta_0) \rightarrow 0$.
We first prove that $\ell_{\eta_n}$ converges to $\ell_{\eta_0}$ pointwise.
Suppose  $\ell_{\eta_n}(x) \nrightarrow \ell_{\eta_0}(x)$ for some
$x\in \DD$. Then we can choose an $\epsilon > 0$ and a subsequence
$m(n)$ such that $m(n) \geq N$ and
$|\ell_{\eta_{m(n)}}(x) - \ell_{\eta_0}(x)| > \epsilon$ for every $n$.
Note that $\ell_\eta$ is continuously differentiable and the
derivative of $\ell_\eta$ is bounded by a continuous function $Q$
uniformly in $\eta\in\cH^N$ by \eqref{eq:Lipsc_cond_reg}.
Thus we can choose a $\delta > 0$ such that
$|\ell_{\eta_{m(n)}}(y) - \ell_{\eta_0}(y)| > \epsilon/2$ for every
$n\geq 1$ and $y$ with $|y-x| < \delta$.
Note that $\delta > 0$ can be chosen sufficiently small so that
$\eta_0(y) > \eta_0(x)/2$ for every $y$ with $|y-x| < \delta$.
Since $\ell_{\eta_{m(n)}}(y) - \ell_{\eta_0}(y) 
= 2\log \sqrt{\eta_{m(n)}(y)/\eta_0(y)}$, there exists a 
$\bar\epsilon > 0$ such that
$\left|1-\sqrt{\eta_{m(n)}(y)/\eta_0(y)}\right| > \bar\epsilon$
for every $n\geq 1$ and $y$ with $|y-x| < \delta$. Note that, 
\[
  h^2(\eta_{m(n)}, \eta_0) \geq \int_{(x-\delta,x+\delta)}
  \bigg( 1 - \sqrt{\frac{\eta_{m(n)}}{\eta_0}} \bigg)^2 dP_{\eta_0}
  \geq \delta\bar\epsilon^2 \eta_0(x) > 0,
\]
for every $n\geq 1$, which contradicts $h(\eta_n, \eta_0) \rightarrow
0$. Conclude that $\ell_{\eta_n}(x) \rightarrow \ell_{\eta_0}(x)$, for
every $x$.

By \eqref{eq:Lipsc_cond_reg}, we have for every sufficiently small $y>0$,
\bean
&& \sup_{\eta\in\cH^N} \left| \int \left[
  \frac{\ell_{\eta}(x+y) - \ell_{\eta}(x)}{y} +
  s_{\eta_0}(x) \right]^2  - 
  \Big(s_{\eta}(x) - s_{\eta_0}(x)\Big)^2 dP_{\eta_0}(x) \right|\\
&=& \sup_{\eta\in\cH^N} \bigg|
  \int \bigg\{ \int_0^1 \Big[s_{\eta}(x+ty) - s_{\eta}(x)\Big] dt
  \\
  && \qquad \qquad \times \left[\frac{\ell_{\eta}(x+y) - \ell_{\eta}(x)}{y} 
  - s_{\eta}(x) + 2s_{\eta_0}(x) \right]\bigg\} dP_{\eta_0}(x) \bigg|\\
&\leq& \sup_{\eta\in\cH^N} y \bigg|
  \int Q(x) \times\left[\frac{\ell_{\eta}(x+y) - \ell_{\eta}(x)}{y} 
  - s_{\eta}(x) + 2s_{\eta_0}(x) \right] dP_{\eta_0}(x) \bigg|\\
  &=& o(1),
\eean
as ${y \downarrow 0}$.
By the Moore-Osgood theorem \cite{taylor2012general}, this enables
us to interchange the two limits in the following equality
\be\begin{split}\label{eq:yd1}
  \lim_{n\rightarrow\infty} P_{\eta_0}(s_{\eta_n} - s_{\eta_0})^2
  &= \lim_{n\rightarrow\infty} \lim_{y \downarrow 0} \int \left[
  \frac{\ell_{\eta_n}(x+y) - \ell_{\eta_n}(x)}{y} +
  s_{\eta_0}(x) \right]^2 dP_{\eta_0}(x)\\
  &= \lim_{y \downarrow 0} \lim_{n\rightarrow\infty} \int \left[
  \frac{\ell_{\eta_n}(x+y) - \ell_{\eta_n}(x)}{y} +
  s_{\eta_0}(x) \right]^2 dP_{\eta_0}(x).
\end{split}\ee
The right-hand side of \eqref{eq:yd1} is equal to 0 by dominated
convergence based on pointwise convergence of $\ell_{\eta_n}$ 
to $\ell_{\eta_0}$.
\qed

\subsubsection{Proof of Lemma \ref{lem:reg_consistency}}\label{subsubsec:3}

Since $\eta_0$ is continuous and $\eta_0(0) > 0$, there exist
constants $C > 0$ and $\delta > 0$
such that $\int_\gamma^\infty \eta_0(x) dx < 1/2 -
C(\gamma\wedge\delta)$ for every $\gamma > 0$.
Let $\epsilon >0$ be a constant such that 
$\epsilon < aC\delta$, where $a^2 = \liminf_n\rho_{\rm min}(\bZ_n)/(2L^2)$.

For a given large enough $n$, fix $\eta\in\cH_n$ with
$h_n\big((\theta,\eta),(\theta_0,\eta_0)\big) < \epsilon$.
Since the Hellinger distance is bounded below by half of the total
variational distance, we have,
\be
\label{eq:yd2}
    h^2(p_{\theta,\eta,i}, p_{\theta_0,\eta_0,i}) \geq
   d_V^2(p_{\theta,\eta,i}, p_{\theta_0,\eta_0,i})/4
  = \sup_B |P_{\theta,\eta,i}(B) - P_{\theta_0,\eta_0,i}(B)|^2. 
\ee
By letting $B=[\theta^T Z_i, \infty)$ in \eqref{eq:yd2}, we have,
\be
\label{eq:yd3}
h^2(p_{\theta,\eta,i}, p_{\theta_0,\eta_0,i})
  \geq \Big( \int_{|(\theta-\theta_0)^T Z_i|}^\infty \eta_0(x) dx
    - \frac{1}{2} \Big)^2
  \geq C^2 \big(|(\theta-\theta_0)^T Z_i|\wedge\delta\big)^2.
\ee
Let $\NN_{\delta,n} = \{i \leq n: |(\theta-\theta_0)^TZ_i| \geq \delta\}$
and let $N_{\delta,n}$ denote its cardinality. Then \eqref{eq:yd3}
implies,
\bea
  \epsilon^2 \geq h_n^2((\theta,\eta),\theta_0,\eta_0)) 
  &\geq& \frac{C^2}{n} \sum_{i=1}^n \big(|(\theta-\theta_0)^T Z_i|
    \wedge \delta\big)^2\nonumber\\
  &\geq& \frac{C^2 N_{\delta,n} \delta^2}{n} 
    + \frac{C^2}{n}\sum_{i\notin\NN_{\delta,n}}
      |(\theta-\theta_0)^TZ_i|^2. \label{eq:yd4}
\eea    
The first term of \eqref{eq:yd4} is greater than
$N_{\delta,n}\epsilon^2/(na^2)$ since $\epsilon< aC\delta$, which
implies $N_{\delta,n}/n< a^2$.
On the other hand, for the second term of \eqref{eq:yd4}, note that,
\[
\sum_{i\notin\NN_{\delta,n}}|(\theta-\theta_0)^TZ_i|^2
\geq \sum_{i=1}^n|(\theta-\theta_0)^TZ_i|^2
- N_{\delta,n} \max_i |(\theta-\theta_0)^TZ_i|^2
\]
Since $\sum_{i=1}^n|(\theta-\theta_0)^TZ_i|^2 \geq n |\theta-\theta_0|^2
\rho_{\rm min}(\bZ_n)$ and $\max_i |(\theta-\theta_0)^TZ_i|^2
\leq  L^2 |\theta-\theta_0|^2$, we have,
\be
\label{eq:yd5}
\frac{C^2}{n}\sum_{i\notin\NN_{\delta,n}}|(\theta-\theta_0)^TZ_i|^2
\geq C^2 |\theta-\theta_0|^2 \Big( \rho_{\rm min}(\bZ_n)
    - L^2 \frac{N_{\delta,n}}{n}\Big). 
\ee
Since $N_{\delta,n}/n< a^2$ and
$a^2=\liminf_n\rho_{\rm min}(\bZ_n) / (2L^2)$, \eqref{eq:yd4} and
\eqref{eq:yd5} together imply $|\theta-\theta_0|^2 \leq K_1
\epsilon^2$, where $K_1 = 2/\big(C^2\rho_{\rm min}(\bZ_n)\big)$.

The proof is complete if we show that $h(\eta, \eta_0) < K \epsilon$
for some constant $K>0$. Note that for every $i$,
\be\label{eq:h_triangle2}
  h^2(\eta,\eta_0) = h^2(p_{\theta,\eta,i}, p_{\theta,\eta_0,i})
  \leq 2\bigl(h^2(p_{\theta,\eta,i}, p_{\theta_0,\eta_0,i})
    + h^2(p_{\theta_0,\eta_0,i}, p_{\theta,\eta_0,i})\bigr).
\ee
In addition, there exists a constant $K_2>0$ such that,
\be \label{eq:h_lips}
  \sup_i h^2(p_{\theta_0,\eta_0,i}, p_{\theta,\eta_0,i})
  \leq K_2 |\theta-\theta_0|^2,
\ee
for every $\theta$ that is sufficiently close to $\theta_0$ because
(denote $\dot \eta_0= d\eta_0/dx$),
\bean
	&& \int \Big(\sqrt{\eta_0(x+y)} - \sqrt{\eta_0(x)}\Big)^2 dx
	= y^2 \int \bigg(\int_0^1 \frac{\dot\eta_0(x+ty)}
	{\sqrt{\eta_0(x+ty)}}dt\bigg)^2 dx \\
	&& ~~ \leq y^2 \int \int_0^1 \bigg(\frac{\dot\eta_0(x+ty)}
	{\eta_0(x+ty)}\bigg)^2 \eta_0(x+ty) \;dt\,dx 
	\leq y^2 P_{\eta_0} Q^2,
\eean
for small enough $y$, where the last inequality holds by Fubini's
theorem and \eqref{eq:Lipsc_cond_reg}. So we have,
\bean
  h^2(\eta,\eta_0) &\leq& \frac{1}{n}\sum_{i=1}^n 2
  \bigl(h^2(p_{\theta,\eta,i}, p_{\theta_0,\eta_0,i})
    + h^2(p_{\theta_0,\eta_0,i}, p_{\theta,\eta_0,i})\bigr) \\
    &\leq& 2 h_n^2 ((\theta,\eta), (\theta_0,\eta_0))
  + 2K_2|\theta-\theta_0|^2,
\eean
where the first inequality holds by \eqref{eq:h_triangle2}
and the second inequality holds by the definition of $h_n$ and
\eqref{eq:h_lips}. Since we have already shown that
$|\theta-\theta_0|^2< K_1 \epsilon^2$, we conclude that
$h(\eta,\eta_0)\le K\epsilon$, where $K=\sqrt{2+2K_1 K_2}$.
\qed

\subsubsection{Proof of (\ref{eq:AB_n_sqrtn})}\label{subsubsec:4}

We start by proving the following two claims: for every
$\widetilde M_n \rightarrow \infty$ with
$\widetilde M_n/\sqrt{n}\rightarrow 0$,
\begin{equation}
\label{eq:lr_apprx_small_condition}
  \sup_{|h| \leq \widetilde M_n} \sup_{\eta\in \cH^N}
  \left|\left( \ell^{(n)}_{\theta_n(h),\eta} - \ell^{(n)}_{\theta_0,\eta}
  - \frac{h^T}{\sqrt{n}} \score_{\theta_0,\eta}^{(n)} \right)^o\right|
  =  o_{P_0}(\widetilde M_n^2),
\end{equation}
and,
\begin{equation}\label{eq:lr_apprx_large_condition}
  \sup_{\widetilde M_n < |h| < \epsilon \sqrt{n}} \sup_{\eta\in\cH^N}
  \left|\left( \ell^{(n)}_{\theta_n(h),\eta} - \ell^{(n)}_{\theta_0,\eta}
  \right)^o\right| \cdot |h|^{-2}
  = o_{P_0}(1),
\end{equation}
for sufficiently small $\epsilon > 0$.

First, the equality,
\[
  \left( \ell^{(n)}_{\theta_n(h),\eta} - \ell^{(n)}_{\theta_0,\eta} 
    - \frac{h^T}{\sqrt{n}} \score_{\theta_0,\eta}^{(n)} \right)^o
  = \frac{h^T}{\sqrt{n}}\int_0^1(\score_{\theta_n(th),\eta}^{(n)}
    -\score_{\theta_0,\eta}^{(n)})^o dt,
\]
implies that the left-hand side of \eqref{eq:lr_apprx_small_condition}
is bounded by,
\be\label{eq:score_diff}
  \sup_{|h| \leq \widetilde M_n}\sup_{\eta\in\cH^N}
    \bigg|\frac{\widetilde M_n}{\sqrt{n}}
  (\score_{\theta_n(h),\eta}^{(n)} -\score_{\theta_0,\eta}^{(n)})^o\bigg|.
\ee
Since,
\bean
  \sup_{|h| \leq \widetilde M_n}\sup_{\eta\in\cH^N}
    \bigg|\frac{1}{\sqrt{n}} (\score_{\theta_n(h),\eta}^{(n)}
      -\score_{\theta_0,\eta}^{(n)})^o\bigg| = O_{P_0}(1),
\eean
by asymptotic tightness of \eqref{eq:score_process_reg}, we
conclude \eqref{eq:score_diff} is of order $o_{P_0}(\widetilde M_n^2)$.

Similarly by the equality,
\[
  \left( \ell^{(n)}_{\theta_n(h),\eta} - \ell^{(n)}_{\theta_0,\eta} \right)^o
  =\frac{h^T}{\sqrt{n}}\int_0^1\Big(\score^{(n)}_{\theta_n(th),\eta}\Big)^o dt,
\]
the left-hand side of \eqref{eq:lr_apprx_large_condition} is bounded by,
\be\label{eq:varying_score}
  \sup_{\widetilde M_n < |h| < \epsilon \sqrt{n}} \sup_{\eta\in\cH_n}
    \bigg|\frac{h^T}{\sqrt{n}}
    \Big(\score^{(n)}_{\theta_n(h),\eta}\Big)^o\bigg| \cdot |h|^{-2}.
\ee
By asymptotic tightness of \eqref{eq:score_process_reg}, 
\[
  \sup_{\widetilde M_n < |h| < \epsilon \sqrt{n}} \sup_{\eta\in\cH_n}
    \bigg|\frac{1}{\sqrt{n}}
    \Big(\score^{(n)}_{\theta_n(h),\eta}\Big)^o\bigg| = O_{P_0}(1),
\]
so \eqref{eq:varying_score} is of order $o_{P_0}(1)$.

Next, we show that for every $C_1 > 0$, there exists a $C_2 > 0$ such that,
\begin{equation}
\label{eq:sqrtn_lower}
  P_0^{(n)} \bigg( \bigg\{\inf_{\eta\in\cH_n} \int_\Theta 
  \frac{p^{(n)}_{\theta, \eta}}{p^{(n)}_{\theta_0, \eta}}d\Pi_\Theta(\theta) 
  \geq C_2 \left(\frac{M_n}{\sqrt{n}}\right)^p e^{-C_1 M_n^2}\bigg\} \bigg)
  \rightarrow 1.
\end{equation}
Let,
\[
  \Phi_n(h,\eta) = \ell^{(n)}_{\theta_n(h),\eta} - \ell^{(n)}_{\theta_0,\eta}
    = \sum_{i=1}^5 A_{n,i}(h,\eta),
\]
where,
\bean
  A_{n,1}(h,\eta) &=& \left( \ell^{(n)}_{\theta_n(h),\eta} 
    - \ell^{(n)}_{\theta_0,\eta} - \frac{h^T}{\sqrt{n}}
      \score_{\theta_0,\eta}^{(n)} \right)^o,\\
  A_{n,2}(h,\eta) &=& \frac{1}{2}h^T (V_{n,\eta_0} - V_{n,\eta})h,\\
  A_{n,3}(h,\eta) &=& \frac{h^T}{\sqrt{n}}
    \left(\score_{\theta_0,\eta}^{(n)} - P_0^{(n)}
    \score_{\theta_0,\eta}^{(n)}\right),\\
  A_{n,4}(h,\eta) &=& - \frac{1}{2} h^T V_{n,\eta_0} h,\\
  A_{n,5}(h,\eta) &=& P_0^{(n)} \Big( \ell^{(n)}_{\theta_n(h),\eta}
    - \ell^{(n)}_{\theta_0,\eta}\Big) + \frac{1}{2} h^T V_{n,\eta} h.
\eean
Note that $\int \exp(\Phi_n(h,\eta)) d\Pi_n(h) \geq
\int_{|h| \leq C_1 M_n}\exp(\Phi_n(h,\eta)) d\Pi_n(h)$, where
$\Pi_n$ is the prior for the centred and rescaled parameter
$h = \sqrt{n}(\theta-\theta_0)$. For $h$ and $\eta\in\cH_n$ with
$|h| \leq C_1 M_n$, the suprema of $|A_{n,1}(h,\eta)|$ and
$|A_{n,2}(h,\eta)|$ are of order $o_{P_0}(M_n^2)$ by
\eqref{eq:lr_apprx_small_condition} and \eqref{eq:V_conti_condition},
respectively. The supremum of $|A_{n,3}(h,\eta)|$ is of the same order
by asymptotic tightness of \eqref{eq:score_process_reg}.
The quantity $|A_{n,4}(h,\eta)|$ is uniformly bounded by
$C_1^2 M_n^2 \|V_{n,\eta_0}\|/2$ and the supremum of
$|A_{n,5}(h,\eta)|$ is of order $o(M_n^2)$ by \eqref{eq:unif_quad_reg}.
Therefore, for $|h| \leq C_1 M_n$ and
$\eta\in\cH_n$, $\Phi_n(h,\eta)$ is uniformly bounded below by,
\[
  M_n^2 \Big(-\frac{C_1^2}{2} \cdot\|V_{n,\eta_0}\| + o_{P_0}(1)\Big).
\]
Thus,
\be
  \begin{split} 
  \int_\Theta
  \frac{p^{(n)}_{\theta, \eta}}{p^{(n)}_{\theta_0, \eta}}&d\Pi_\Theta(\theta)
  \geq \int_{|h| \leq C_1 M_n} \exp(\Phi_n(h,\eta)) d\Pi_n(h)\\ 
  &\geq \int_{\sqrt{n}|\theta-\theta_0| \leq C_1 M_n} 
    \exp\bigg[M_n^2 \Big(-\frac{C_1^2}{2} \cdot\|V_{n,\eta_0}\| 
    + o_{P_0}(1)\Big)\bigg] d\Pi_\Theta(\theta).
  \end{split}
\ee
Also, the thickness of $\Pi_\Theta$ at $\theta_0$ implies that,
\[
  \Pi_\Theta\{\theta: \sqrt{n}|\theta-\theta_0| \leq C_1 M_n\}
  \geq C_2 (M_n/\sqrt{n})^p,
\] 
for some $C_2>0$.
Since $\limsup_n\rho_{\max}(V_{n,\eta_0}) < \infty$ by
\eqref{eq:V_positive}, and $C_1 > 0$ is arbitrary, we conclude that
\eqref{eq:sqrtn_lower} holds.

Finally, we prove that there exist $C > 0$ and $\epsilon > 0$ such that,
\begin{equation}
\label{eq:sqrtn_upper}
  P_0^{(n)} \bigg( \sup_{M_n < |h| < \epsilon \sqrt{n}}\sup_{\eta\in\cH_n}
    \frac{p^{(n)}_{\theta_n(h), \eta}}{p^{(n)}_{\theta_0, \eta}} e^{C |h|^2}
    \leq 1 \bigg) \rightarrow 1.
\end{equation}
For given $\delta > 0$, by \eqref{eq:unif_quad_reg}, there exists an
$\epsilon > 0$ such that
\be\label{eq:quad_delta}
  \sup_{\eta\in\cH^N} \left|P_0^{(n)} \Big(
    \ell^{(n)}_{\theta_n(h),\eta}
      - \ell^{(n)}_{\theta_0,\eta}\Big)
    + \frac{1}{2} h^T V_{n,\eta} h\right| < \delta \cdot |h|^2,
\ee
for every $h$ with $|h| < \sqrt{n}\epsilon$. 
Write,
\be\label{eq:expand}
  \log\frac{p^{(n)}_{\theta_n(h), \eta}}{p^{(n)}_{\theta_0, \eta}}
  = \sum_{i=1}^4 B_{n,i}(h,\eta),
\ee
where,
\bean
  B_{n,1}(h,\eta) &=& \left( \ell^{(n)}_{\theta_n(h),\eta}
    - \ell^{(n)}_{\theta_0,\eta} \right)^o,\\
  B_{n,2}(h,\eta) &=& P_0^{(n)} \Big( \ell^{(n)}_{\theta_n(h),\eta}
    - \ell^{(n)}_{\theta_0,\eta}\Big) + \frac{1}{2}h^TV_{n,\eta}h,\\
  B_{n,3}(h,\eta) &=& \frac{1}{2}h^T (V_{n,\eta_0} - V_{n,\eta})h,\\
  B_{n,4}(h,\eta) &=& - \frac{1}{2} h^T V_{n,\eta_0} h.
\eean
For $M_n < |h| < \epsilon\sqrt{n}$ and $\eta\in\cH_n$,
$|B_{n,1}(h,\eta)|$ and $|B_{n,3}(h,\eta)|$ are bounded by
$|h|^2 \times o_{P_0}(1)$ by \eqref{eq:lr_apprx_large_condition}
and \eqref{eq:V_conti_condition}, respectively, where the $o_{P_0}(1)$
term does not depend on $h$ and $\eta$. Furthermore,
$|B_{n,2}(h,\eta)| \leq \delta |h|^2$ by \eqref{eq:quad_delta}, and
$B_{n,4}(h,\eta) \leq -\rho_{\min}(V_{n,\eta_0})|h|^2/2$.
Thus, \eqref{eq:expand} is bounded above by,
\[
  |h|^2 \cdot \Big( -\frac{1}{2} \rho_{\min}(V_{n,\eta_0})
  + \delta + o_{P_0}(1) \Big),
\]
for every $h$ with $|h| < \sqrt{n}\epsilon$ and $\eta\in\cH_n$.
Since $\delta > 0$ can be arbitrarily small and
$\liminf_n\rho_{\min}(V_{n,\eta_0}) > 0$ by
\eqref{eq:V_positive}, we conclude that \eqref{eq:sqrtn_upper}
holds for $C < \liminf_n\rho_{\min}(V_{n,\eta_0})/2$.
\qed

\subsubsection{Proof of \eqref{eq:symm_bd}}\label{subsubsec:5}

For the first inequality of \eqref{eq:symm_bd}, note that,
\[
  |\sqrt{a_1+a_2} - \sqrt{b_1 + b_2}|
  \leq |\sqrt{a_1} - \sqrt{b_1}| + |\sqrt{a_2} - \sqrt{b_2}|,
\]
and $(a_1+b_1)^2 \leq 2(a_1^2 + b_1^2)$ for any $a_1, a_2, b_1, b_2\geq0$.
Thus,
\bean
  h^2(\bar p, \bar q) &=& \int
  \bigg( \sqrt{\frac{p+p^-}{2}} - \sqrt{\frac{q+q^-}{2}} \bigg)^2 d\mu\\
  &\leq& \int (\sqrt{p} - \sqrt{q})^2 + (\sqrt{p^-} - \sqrt{q^-})^2 d\mu
    = 2 h^2(p,q),
\eean
and so $h(\bar p, \bar q) \leq \sqrt 2 h(p,q)$ for any two densities $p$ and
$q$ supported on $\DD$. 

For the second and third inequalities of \eqref{eq:symm_bd}, we may
assume that $p$ is symmetric. Then,
\bean
  K(\bar p,\bar q) &=& K(p,\bar q) 
  = \int \bigg(\log p - \log\Big(\frac{q+q^-}{2}\Big)\bigg) dP
  \\
  &\leq& \int \bigg( \log p - \frac{1}{2}\Big\{\log \frac{q}{2}
  + \log \frac{q^-}{2} \Big\}\bigg) dP
  = \int \frac{1}{2}\Big(\log\frac{p}{q} + \log\frac{p}{q^-}\Big) dP,
\eean
where the inequality holds by the concavity of $x \mapsto \log(x)$. 
Also, the symmetry of $p$ implies that
$\int \log(p/q^-) dP = \int \log(p/q) dP$ and so 
$K(\bar p, \bar q) \leq K(p,q)$.
In addition,
\[
  V(\bar p, \bar q) = V(p, \bar q)
  \leq \int \bigg(\log p - \log\Big(\frac{q + q^-}{2}\Big)\bigg)^2 dP
  = \int \bigg(\log\frac{2p}{q+q^-}\bigg)^2 dP.
\]
Since $p/q \wedge p/q^- \leq 2p/(q+q^-) \leq p/q\vee p/q^-$,
we have,
\[
  \bigg|\log\frac{2p}{q+q^-}\bigg|
  \leq \bigg|\log\frac{p}{q} \bigg| + \bigg|\log\frac{p}{q^-} \bigg|,
\]
and so $V(\bar p, \bar q) \leq 4 \int \big(\log(p/q)\big)^2 dP
= 4\big( V(p,q) + K^2 (p,q)\big)$.
\qed

\subsubsection{Proof of \eqref{eq:prod_bd}}\label{subsubsec:6}

Assume that $\epsilon > 0$ is sufficiently small and
$|\theta_1-\theta_2|\vee |\theta-\theta_0| < \epsilon$.
Using \eqref{eq:lipschitz_bd_lem} and the fact that
$(a+b)^2\leq 2(a^2+b^2)$ for all $a,b\in\RR$, we have
the second and third inequalities of \eqref{eq:prod_bd} because,
\bean
  K(p_{\theta_0,\eta_0,i}, p_{\theta,\eta,i})
  &=& \int (\ell_{\theta_0,\eta_0,i} - \ell_{\theta_0,\eta,i})
  + (\ell_{\theta_0,\eta,i} - \ell_{\theta,\eta,i}) dP_{\theta_0,\eta_0,i}
  \\
  &\lesssim& K(\eta_0,\eta) + |\theta-\theta_0|,
\eean
and,
\bean
  V(p_{\theta_0,\eta_0,i}, p_{\theta,\eta,i})
  &\leq& 2\int (\ell_{\theta_0,\eta_0,i} - \ell_{\theta_0,\eta,i})^2
    + (\ell_{\theta_0,\eta,i}
    - \ell_{\theta,\eta,i})^2 dP_{\theta_0,\eta_0,i}\\
  &\lesssim& V(\eta_0,\eta) + K^2(\eta_0,\eta)
    + |\theta-\theta_0|^2,
\eean
for every $\eta\in\cH_0$.

For the first inequality of \eqref{eq:prod_bd},
\bean
  &&h(p_{\theta_1,\eta_1,i}, p_{\theta_2,\eta_2,i})
  \leq h(p_{\theta_1,\eta_1,i}, p_{\theta_1,\eta_2,i})
    + h(p_{\theta_1,\eta_2,i}, p_{\theta_2,\eta_2,i})
    \\
  && ~~ = h(\eta_1,\eta_2)
    + h(p_{\theta_1,\eta_2,i}, p_{\theta_2,\eta_2,i})
  \lesssim h(\eta_1,\eta_2) + |\theta_1-\theta_2|,
\eean
for every $\eta_1,\eta_2 \in \cH_0$, where the last
inequality holds because, with $\dot\eta(x) = d\eta(x)/dx$ and $y_i
= |(\theta_1-\theta_2)^T Z_i|$,
\bean
	h^2(p_{\theta_1,\eta,i}, p_{\theta_2,\eta,i})
	&=& \int \Big(\sqrt{\eta(x+y_i)} - \sqrt{\eta(x)}\Big)^2 dx
	\\	
	&=& y_i^2 \int \bigg(\int_0^1 \frac{\dot\eta(x+ty_i)}
	{\sqrt{\eta(x+ty_i)}}dt\bigg)^2 dx 
	\\
	&\leq& y_i^2 \int \int_0^1 \bigg(\frac{\dot\eta(x+ty_i)}
	{\eta(x+ty_i)}\bigg)^2 \eta(x+ty_i) \;dt\,dx 
	\leq y_i^2 P_{\eta} \widetilde Q^2,
\eean
for every $\eta\in\cH_0$.
\qed

\subsubsection{Proof of the asymptotic tightness of%
  \eqref{eq:score_process2}}
\label{subsubsec:7}

Without loss of generality we may assume that $\theta_0$ is equal
to the zero vector. For given $a \in \RR^p$, let,
\bean
  Z_{ni}(\theta,\eta) = a^T\score_{\theta,\eta,i}/\sqrt{n}, ~~~~~
  S_{ni} = \sup_{\theta \in B_\epsilon}\sup_{\eta\in\cH_0}
    |Z_{ni}(\theta,\eta)|,
\eean
and $\cF = B_\epsilon\times\cH_0$. Let $N^n_{[]} (\delta,\cF)$ be the
minimal number of sets $N$ in a partition
$\{\cF_j: 1\leq j \leq N\}$ of $\cF$ such that,
\be\label{eq:bracket_def}
  \sum_{i=1}^n P_0 \sup_{\substack{(\theta_1,\eta_1) \in \cF_j \\
  (\theta_2,\eta_2) \in \cF_j}} \Big|Z_{ni}(\theta_1,\eta_1) - Z_{ni}
    (\theta_2,\eta_2) \Big|^2 \leq \delta^2,
\ee
for every $j\leq N$. The bracketing central limit theorem
(Theorem~2.11.9 of \cite{van1996weak}) assures that if ,
\be
\label{eq:bracket_clt}
  \begin{split}
  &\sum_{i=1}^n P_0\big( S_{ni} 1_{\{S_{ni} > \gamma\}}\big) = o(1)
    \quad\text{for every $\gamma > 0$},\\
  &\int_0^{\delta_n}\sqrt{\log N^n_{[]} (\delta,\cF)}\,d\delta < \infty,
    \quad\text{for every $\delta_n \downarrow 0$},
  \end{split}
\ee
then \eqref{eq:score_process2} is asymptotically tight.

Since $|Z_i|$'s are bounded and the mean probability $H$ of the
Dirichlet process is compactly supported, there exist functions
$Q_j$ for $j=1,2$, such that $Q_j(x) = C_j(1 + |x|^j)$ for some
constants $C_j > 0$, and,
\be
  |\ell_\eta(x+y) - \ell_\eta(x)| \leq |y| \cdot Q_1(x),\quad
  |s_\eta(x+y) - s_\eta(x)| \leq |y| \cdot Q_2(x),
\ee
for every $\eta\in\cH_0$, $x$ and $y$ with $|y| \leq L\epsilon$
(see Lemma~3.2.3 of \cite{chae2015semiparametric} for details).
Thus $\sqrt{n} |Z_{ni}(\theta,\eta)| \lesssim Q_1(X_i)$ for every
$i\leq n$ and $(\theta,\eta) \in \cF$. Since $Q_1$ is
$P_{\eta_0}$-square-integrable,
\be
  \sum_{i=1}^n P_0\big( S_{ni} 1_{\{S_{ni} > \gamma\}}\big)
  \leq \sqrt{n} P_{\eta_0}\big( Q_1 1_{\{Q_1 > \sqrt{n} \gamma\}}\big)
  \leq \gamma^{-1} P_{\eta_0} Q_1^2 1_{\{Q_1 > \sqrt{n} \gamma\}} = o(1),
\ee
for every $\gamma > 0$, so the first condition of
\eqref{eq:bracket_clt} is satisfied.

Note that,
\begin{equation}
  \begin{split} \label{eq:z_ni_tri}
  |Z_{ni}(\theta_1,\eta_1) -& Z_{ni}(\theta_2,\eta_2)|\\
  & \leq |Z_{ni}(\theta_1,\eta_1) - Z_{ni}(\theta_2,\eta_1)|
      + |Z_{ni}(\theta_2,\eta_1) - Z_{ni}(\theta_2,\eta_2)|.
  \end{split}
\end{equation}
The first term of the right-hand side of \eqref{eq:z_ni_tri} is
bounded by,
\be\label{eq:zni_theta}
  \sup_{\eta\in\cH_0}|Z_{ni}(\theta_1,\eta) - Z_{ni}(\theta_2,\eta)|
  \lesssim \sup_{\eta \in \cH_0} \frac{1}{\sqrt{n}} 
    \big| \score_{\theta_1,\eta,i} - \score_{\theta_2,\eta,i} \big| 
  \lesssim \frac{|\theta_1-\theta_2|}{\sqrt{n}} Q_2(X_i).
\ee
For every $y$ with $|y| \leq L\epsilon$, let
$\cS_y = \{x \mapsto s_\eta(x-y): \eta\in\cH_0\}$.
Since the first and second derivatives of $x\mapsto s_\eta(x-y)$ are
of order $O(x^2)$ and $O(x^3)$, (uniformly in $|y|\leq L\epsilon$ and
$\eta\in\cH_0$) and $\eta_0(x) = O(e^{-Cx^2})$ for some $C > 0$ as
$|x| \rightarrow \infty$, we have
$\sup_{|y| \leq L\epsilon} \log N_{[]}(\delta,\cS_y, L_2(P_{\eta_0}))
\lesssim \delta^{-1/2}$ for every small enough $\delta > 0$ by
Corollary~2.7.4 of \cite{van1996weak} with $\alpha=r=2$, $d=1$,
$V=1/2$ and a partition $\RR = \cup_{j=-\infty}^\infty [j-1,j)$.
Assume that some sufficiently small $\delta > 0$ is given and we
choose a sequence $(y_j)_{j=0}^{N_\delta}$ such that
$-\epsilon L = y_0 < y_1 < \cdots < y_{N_\delta} = \epsilon L$
and $y_{j+1} - y_j < \delta$. Since $N_\delta \lesssim \delta^{-1}$
and $\log N_{[]}(\delta^{3/2},\cS_y, L_2(P_{\eta_0})) \lesssim
\delta^{-3/4}$, we can construct a partition
$\{\cH_l: 1\leq l\leq \overline N_\delta\}$ of $\cH_0$ by taking
all intersections of sets in $N_\delta+1$ partitions, so that
$\log\overline N_\delta \leq N_\delta \cdot
\log N_{[]}(\delta^{3/2},\cS_y, L_2(P_{\eta_0})) \lesssim
\delta^{-7/4}$ and,
\[
  \int \sup_{\eta_1,\eta_2\in\cH_l} |s_{\eta_1}(x-y_j)
    - s_{\eta_2}(x-y_j)|^2 dP_{\eta_0}(x) \leq \delta^3,
\]
for every $l$ and $j$. Applying Lemma~2.2.2 of \cite{van1996weak}
with $\psi(x) = x^2$, we have,
\be\label{eq:s_eta_y}
  \int \max_{1\leq j \leq N_\delta} \sup_{\eta_1,\eta_2\in\cH_l}
  |s_{\eta_1}(x-y_j) - s_{\eta_2}(x-y_j)|^2 dP_{\eta_0}(x)
  \lesssim \delta^2,
\ee
for every $l$.

Now, consider the second term of the right-hand side of
\eqref{eq:z_ni_tri}. For every $\theta\in B_\epsilon$ and $i\geq 1$,
we can choose $j$ such that $|\theta^T Z_i - y_j| \leq \delta$.
Then,
\bean
  \big(Z_{ni}(\theta,\eta_1) - Z_{ni}(\theta,\eta_2)\big)^2
  &\lesssim& \frac{1}{n} \big|\score_{\theta,\eta_1,i}
    - \score_{\theta,\eta_2,i}\big|^2\\
  &\lesssim& \frac{1}{n} |s_{\eta_1}(X_i - \theta^T Z_i)
    - s_{\eta_2}(X_i - \theta^T Z_i)|^2\\
  &\lesssim& \frac{\delta^2}{n} Q^2_2(X_i) +
    \frac{1}{n}|s_{\eta_1}(X_i - y_j) - s_{\eta_2}(X_i - y_j)|^2,
\eean
so we have,
\be\label{eq:zni_eta}
  P_0 \bigg( \sum_{i=1}^n \sup_{\theta\in B_\epsilon}
    \sup_{\eta_1,\eta_2\in\cH_l} \big(Z_{ni}(\theta,\eta_1)
    - Z_{ni}(\theta,\eta_2)\big)^2 \bigg) \lesssim \delta^2,
\ee
for every $l$.

Finally, the two bounds \eqref{eq:zni_theta} and \eqref{eq:zni_eta}
combined with \eqref{eq:z_ni_tri}, imply that,
\[
  \sum_{i=1}^n P_0 \sup_{\substack{|\theta_1-\theta_2| \leq \delta \\
  \eta_1,\eta_2 \in \cH_l}} \Big|Z_{ni}(\theta_1,\eta_1)
    - Z_{ni}(\theta_2,\eta_2) \Big|^2 \lesssim \delta^2,
\]
for every $l$. Since $N(\delta, B_\epsilon, |\cdot|) \lesssim
\delta^{-p}$, a partition satisfying \eqref{eq:bracket_def} can be
constructed by product sets of each partition of $B_\epsilon$ and
$\cH_0$, the order of which is bounded as (for some constant $K > 0$),
\be \label{eq:bracketing_dpm}
  \log N^n_{[]}(\delta,\cF) \lesssim
  \log \overline N_{K\delta} + \log \delta^{-p} \lesssim \delta^{-7/4},
\ee
so the second condition of \eqref{eq:bracket_clt} is satisfied.
\qed

\subsubsection{Proof of asymptotic tightness in Corollary%
  \ref{cor:ex-series}}
\label{subsubsec:8}

We follow the steps of the proof of asymptotic tightness in
Corollary~\ref{cor:ex-dpm}. Without loss of generality we assume that
$\theta_0=0$, and define $Z_{ni}(\theta,\eta)$, $S_{ni}$,
$\cF$ and $N_{[]}^n(\delta,\cF)$ as in the proof of
Corollary~\ref{cor:ex-dpm}. The first condition of
\eqref{eq:bracket_clt} is proved by replacing $Q_j$'s as constant
functions. Inequalities \eqref{eq:z_ni_tri} and \eqref{eq:zni_theta} 
are shown to hold in the same way.

Let $\cS = \{x\mapsto s_\eta(x) : \eta\in\cH_0\}$.
Applying Theorem~2.7.1 of \cite{van1996weak} with $\alpha=d=1$, we
have $\log N(\delta, \cS, \|\cdot\|_\infty) \lesssim \delta^{-1}$.
This implies that there exists a partition
$\{\cH_l: 1 \leq l \leq \overline N_\delta\}$ of $\cH_0$ such
that $\overline N_\delta \lesssim \delta^{-1}$ and,
\[
  \sup_{\eta_1,\eta_2 \in \cH_l} \sup_{x\in\mathbb{D}}
  |s_{\eta_1}(x) - s_{\eta_2}(x)| < \delta,
\]
for every $l$. Thus, \eqref{eq:zni_eta} holds.
Replacing the entropy bound \eqref{eq:bracketing_dpm} by,
\[
  \log N^n_{[]}(\delta,\cF) \lesssim \log \overline N_{K\delta}
    + \log \delta^{-p} \lesssim \delta^{-1},
\]
we follow the remainder of the proof of Corollary~\ref{cor:ex-dpm}.
\qed

\subsection{Proofs for Section \ref{sec:lm}}\label{ssec:2}

\subsubsection{Proof of \eqref{eq:unif_quad_lm}}\label{subsubsec:4-1}
Since,
\bean
  && \sup_{\eta\in\cH^N} \left| \frac{1}{n}P_0^{(n)}
    \Big( \ell^{(n)}_{\theta,\eta} - \ell^{(n)}_{\theta_0,\eta}\Big)
      + \frac{1}{2} (\theta-\theta_0)^T V_{n,\eta} (\theta-\theta_0)\right|\\
  && \le \sup_{i,\eta}
    \bigg| P_0 \bigg( \log \frac{\psi_\eta(X_i - Z_i^T \theta| W_i)}
      {\psi_\eta(X_i - Z_i^T \theta_0| W_i)} \bigg)
    + \frac{1}{2} (\theta-\theta_0)^T Z_i 
      v_{\eta}(W_i) Z_i^T (\theta-\theta_0) \bigg|,
\eean
where $i$ runs over the integers and $\eta$ over $\cH^N$,
it suffices to show that,
\begin{equation} \label{eq:uql_multi}
  \sup_w \sup_{\eta \in \cH^N}
    \Big| \int \log \frac{\psi_\eta(x-y|w)}{\psi_\eta(x|w)} 
      d\Psi^w_{\eta_{0}}(x) + \frac{1}{2} y^T v_{\eta}(w) y \Big|
  = o(|y|^2),
\end{equation}
as $|y| \rightarrow 0$.

Let $A=\{x=(x_1,\ldots,x_m):x_1>0\}$ and $A^-=\{x:-x\in A\}$.
Note that,
\[
  \psi_\eta(x|w) = \psi_\eta(-x|w) = \psi_\eta(x|-w) = \psi_\eta(-x|-w),
\]
by the symmetry of $f$ and $G$.
Thus, for $\eta\in\cH^N$,
\be \begin{split} \label{eq:eq123} 
  & \int \log \frac{\psi_\eta(x-y|w)}{\psi_\eta(x|w)} d\Psi^w_{\eta_{0}}(x)
    =\int \log \frac{\psi_\eta(x-y|w)}{\psi_\eta(x|w)}
      \psi_{\eta_{0}}(x|w) d\mu(x)\\
  & = \int_{A^-} \log \frac{\psi_\eta(x-y/2|w)}{\psi_\eta(x+y/2|w)}
      \psi_{\eta_0}(x+y/2|w) d\mu(x)\\
  & \qquad+ \int_{A} \log \frac{\psi_\eta(x-y/2|w)}{\psi_\eta(x+y/2|w)}
      \psi_{\eta_0}(x+y/2|w) d\mu(x)\\
  & = \int_A \log \frac{\psi_\eta(-x-y/2|w)}{\psi_\eta(-x+y/2|w)}
      \psi_{\eta_0}(-x+y/2|w) d\mu(x)\\
  & \qquad + \int_A \log \frac{\psi_\eta(x-y/2|w)}{\psi_\eta(x+y/2|w)}
      \psi_{\eta_0}(x+y/2|w) d\mu(x)\\
  & = -\int_A \left[ \ell_{\eta}\Big(x - \frac{y}{2} \Big|w\Big) - 
      \ell_{\eta} \Big(x + \frac{y}{2} \Big|w\Big) \right]\\
  & \qquad \qquad \times \left[ {\psi_{\eta_0}}\Big(x-\frac{y}{2}\Big|w \Big) 
	- {\psi_{\eta_0}}\Big(x+\frac{y}{2} \Big|w \Big) \right] d\mu(x).
\end{split} \ee
The last integral of \eqref{eq:eq123} is equal to,
\[
  \begin{split}
  - \int_0^1\int_0^1\int_A y^T s_{\eta}(x &+ r(y,t)|w)
    s_{\eta_0}^T (x+ r(y,s)|w)y\\
  &\times\psi_{\eta_0}(x+r(y,s)|w)\,d\mu(x) dt ds,
  \end{split}
\]
by Taylor expansion, where $r(y,t) = (t-1/2) y$.
Since, 
\[
  v_{\eta}(w) 
  =2\int_A s_\eta(x|w)\; s_{\eta_0}^T(x|w) d\Psi_{\eta_0}^w(x),
\]
the left-hand side of \eqref{eq:uql_multi}, for fixed $w$ and $\eta$,
is equal to,
\[
\begin{split}
  - y^T \bigg\{\int_0^1\int_0^1 \int_A
    & \Big[ s_{\eta}(x+ r(y,t)|w) s_{\eta_0}^T (x+ r(y,s)|w)
      \psi_{\eta_0}(x+ r(y,s)|w)\\
    & - s_{\eta}(x|w) s_{\eta_0}^T (x|w) \psi_{\eta_0}(x|w) \Big] \;
      d\mu(x) dt\, ds \bigg\}y.
\end{split}
\]
The integrand of the last display is equal to
$A_\eta(x,y,w) + B_\eta(x,y,w)$, where,
\[
\begin{split}
  A_\eta(x,y,w)& = s_{\eta}(x+ r(y,t)|w)\\
  & \times \Big\{ s_{\eta_0}^T (x+ r(y,s)|w) \psi_{\eta_0}(x+r(y,s)|w)
    - s_{\eta_0}^T (x|w) \psi_{\eta_0}(x|w) \Big\},
\end{split}
\]
and,
\[
  B_\eta(x,y,w) = \Big\{s_{\eta}(x+ r(y,t)|w) - s_{\eta}(x|w)\Big\}
    s_{\eta_0}^T (x|w) \psi_{\eta_0}(x|w),
\]
(dependence on $t$ and $s$ is abbreviated for simplicity).
Let $g_{\eta,j}(x|w) = \partial\ell_\eta(x|w) / \partial x_j$ and
$e_j$
be the $j$th unit vector in $\RR^m$.
By \eqref{eq:Lipsc_cond_lm}, it is easy to prove that,
\[
\sup_w \sup_{\eta\in\cH^N} \sup_{t,s\in[0,1]} \int_A
  |e_i^T B_\eta(x,y,w) e_j| d\mu(x) = o(1),
\]
as $|y| \rightarrow 0$. Also, by \eqref{eq:Lipsc_cond_lm},
\[
  \bigg|e_i^T \bigg[\frac{\partial (s_{\eta_0}
    \psi_{\eta_0})}{\partial x} (x|w)
     \bigg] e_j \bigg|
  \leq (Q+ Q^2)(x,w) \psi_{\eta_0}(x|w), 
\]
for every $i,j \leq m$.
Thus, $|e_i^T A_\eta(x,y,w) e_j|$ is bounded by,
\[
\begin{split}
  & |y| \cdot |g_{\eta,i}(x+r(y,t)|w)|
    \cdot\int_0^1 (Q + Q^2) (x+r(y,s)u, w) \psi_{\eta_0}(x+r(y,s)u|w) du\\
  &\leq |y| (1+|y|) \cdot\int_0^1 \Big\{ (Q^2+Q^3)(x+r(y,s)u, w)
    \psi_{\eta_0}(x+r(y,s)u|w) \Big\}du,
\end{split}
\]
where the inequality in the second line holds because,
\[
\begin{split}
  &|g_{\eta,i}(x+r(y,t)|w)|\\
  &\leq |g_{\eta,i}(x+r(y,t)|w) - g_{\eta,i}(x+r(y,s)u|w)|
    + |g_{\eta,i}(x+r(y,s)u|w)|\\
  &\leq (1+|y|) Q(x+r(y,s)u,w).
\end{split}
\]
Therefore,
\[
\begin{split}
  \sup_{s,t\in[0,1]} \sup_{\eta\in\cH^N} \int_A 
    | e_i^T A_\eta(x,y,w) e_j| d\mu(x)dt ds\\
  \leq |y| (1+|y|)\int (Q^2+Q^3)(x,w) d\Psi^w_{\eta_0}(x),
\end{split}\]
which is $o(1)$, uniformly in $w$, as $|y| \rightarrow 0$.
\qed

\subsubsection{Proof of \eqref{eq:d_w_consistency}}
\label{subsubsec:4-2}

To prove \eqref{eq:d_w_consistency}, it suffices to show,
\bean
  \lim_{n\rightarrow\infty} \sup_{\eta\in\cH_n} d_{W_i}(\eta,\eta_0) = 0,
\eean
for every $i\geq 1$ because $W_i$ is contained in a compact set,
\eqref{eq:design} holds, and \eqref{eq:w_class} is uniformly
equicontinuous (note that equicontinuity on a compact domain is
equivalent to uniform equicontinuity).
For given $i\geq 1$, since $\sup_{\eta\in\cH^N} d_{W_i}(\eta,\eta_0) <
\infty$ by \eqref{eq:Lipsc_cond_lm}, we can choose $\eta_n\in\cH_n$,
for large enough $n$ such that,
\[
  \sup_{\eta\in\cH_n} d_{W_i}(\eta,\eta_0) < d_{W_i}(\eta_n,\eta_0) + n^{-1}.
\]
Note that $h_n(\eta_n, \eta_0) \rightarrow 0$ by the definition of $\cH_n$.
Since,
\[
  h^2_n(\eta_n,\eta_0) = \frac{1}{n} \sum_{j=1}^n
  h^2(\psi_{\eta_n}(\cdot|W_j), \psi_{\eta_0}(\cdot|W_j)),
\]
$W_j$ is contained in a compact set, \eqref{eq:design} holds, and
\eqref{eq:w_class} is uniformly equicontinuous, we have
$\lim_{n\rightarrow\infty} h(\psi_{\eta_n}(\cdot|W_j),
\psi_{\eta_0}(\cdot|W_j)) = 0$ for every $j \geq 1$.
Thus, it suffices to show that $d_{W_i}(\eta_n,\eta_0) \rightarrow 0$.
For simplicity, we write $W_i=w$ in the remainder of this proof.

We first prove that $\lim_{n\rightarrow\infty} \ell_{\eta_n}(x|w) =
\ell_{\eta_0}(x|w)$ for every $x$. Suppose  $\ell_{\eta_n}(x|w)
\nrightarrow \ell_{\eta_0}(x|w)$ for some $x$. Then we can choose
an $\epsilon > 0$ and a subsequence $m(n)$ such that $m(n) \geq N$
and $|\ell_{\eta_{m(n)}}(x|w) - \ell_{\eta_0}(x|w)| > \epsilon$ for every $n$.
Note that $x\mapsto\ell_\eta(x|w)$ is continuously differentiable and
its derivative is bounded componentwise by a continuous function
$x \mapsto Q(x,w)$ uniformly in $\eta\in\cH^N$ by \eqref{eq:Lipsc_cond_lm}.
Thus we can choose a $\delta > 0$ such that
$|\ell_{\eta_{m(n)}}(y|w) - \ell_{\eta_0}(y|w)| > \epsilon/2$ for
every $n\geq 1$ and a $y$ with $|y-x| < \delta$. Note that $\delta > 0$
can be chosen sufficiently small so that
$\psi_{\eta_0}(y|w) > \psi_{\eta_0}(x|w)/2$ for every $y$ with
$|y-x| < \delta$. Since,
\[
  \ell_{\eta_{m(n)}}(y|w) - \ell_{\eta_0}(y|w)
  = 2\log \sqrt{\psi_{\eta_{m(n)}}(y|w)/ \psi_{\eta_0}(y|w)},
\]
there exists a $\bar\epsilon > 0$ such that,
\[
  \left|1-\sqrt{\psi_{\eta_{m(n)}}
    (y|w)/\psi_{\eta_0}(y|w)}\right|
  > \bar\epsilon,
\]
for every $n\geq 1$ and $y$ with $|y-x| < \delta$. Since,
\bean
  h^2(\psi_{\eta_{m(n)}}(\cdot|w), \psi_{\eta_0}(\cdot|w)) \
  &\geq& \int_{\{y:|y-x| < \delta\}} \bigg( 1 - 
    \sqrt{\frac{\psi_{\eta_{m(n)}}}{\psi_{\eta_0}}(y|w)} \bigg)^2
    d\Psi^w_{\eta_0}(y)\\
  &\geq& \bar\epsilon^2 \int_{\{y:|y-x| < \delta\}}
    \frac{\psi_{\eta_0}(x|w)}{2}\; dy \geq \gamma,
\eean
for some $\gamma > 0$ and every $n\geq 1$, the above contradicts the
fact that $h(\psi_{\eta_n}, \psi_{\eta_0}) \rightarrow 0$, so 
we conlude that $\ell_{\eta_n}(x|w) \rightarrow \ell_{\eta_0}(x|w)$
for all $x$.

Let $e_j$ be the $j$th unit vector in $\RR^m$ and
$g_{\eta,j}(x|w) = \partial\ell_\eta(x|w) / \partial x_j$.
Then as $y\rightarrow 0$ in $\RR$,
\[
\begin{split}
  \sup_{\eta\in\cH^N} &\bigg| \int 
    \bigg(\frac{\ell_{\eta}(x+y e_j|w) - \ell_{\eta}(x|w)}{y}
      - g_{\eta_0,j}(x|w)\bigg)^2\\
  & \qquad - \Big(g_{\eta,j}(x|w) - g_{\eta_0,j}(x|w)\Big)^2
      d\Psi^w_{\eta_0}(x) \bigg|\\
  & = \sup_{\eta\in\cH^N} \bigg|
    \int \bigg\{\int_0^1 | g_{\eta,j}(x+tye_j | w) - g_{\eta,j}(x|w) |\; dt\\
  & \qquad \times
    \bigg[\frac{\ell_{\eta}(x+y e_j|w) - \ell_{\eta}(x|w)}{y}
      - g_{\eta,j}(x|w) + 2g_{\eta_0,j}(x|w) \bigg] \bigg\}
    d\Psi_{\eta_0}^w(x) \bigg|\\
  & \leq |y| \int Q(x,w) \bigg|\frac{\ell_{\eta}(x+y e_j|w)
    - \ell_{\eta}(x|w)}{y} - g_{\eta,j}(x|w)\\
  & \qquad+ 2g_{\eta_0,j}(x|w)\bigg| d\Psi_{\eta_0}^w(x)\\
  &= o(1),
\end{split}
\]
where the last line holds by \eqref{eq:Lipsc_cond_lm}.
The Moore-Osgood theorem enables the
interchange of the two limits in the following equality:
\be
  \begin{split}\nonumber
  &\lim_{n\rightarrow\infty} \int \Big\{ 
    g_{\eta_n,j}(x|w) - g_{\eta_0,j}(x|w) \Big\}^2 d\Psi_{\eta_0}^w(x)\\
  &= \lim_{n\rightarrow\infty} \lim_{y \rightarrow 0} \int \bigg\{
    \frac{\ell_{\eta_n}(x+ y e_j|w) - \ell_{\eta_0}(x|w)}{y} 
      - g_{\eta_0,j}(x|w) \bigg\}^2 d\Psi_{\eta_0}^w(x)\\
  &= \lim_{y \rightarrow 0} \lim_{n\rightarrow\infty} \int \bigg\{
    \frac{\ell_{\eta_n}(x+ y e_j|w) - \ell_{\eta_0}(x|w)}{y}
      - g_{\eta_0,j}(x|w) \bigg\}^2 d\Psi_{\eta_0}^w(x)\\
  & = \int \Big\{ g_{\eta_0,j}(x|w) - g_{\eta_0,j}(x|w) \Big\}^2 
    d\Psi_{\eta_0}^w(x) = 0.
\end{split}\ee
Conclude that $d_w(\eta_n, \eta_0) = o(1)$.
\qed

\subsubsection{Proof of Lemma \ref{lem:lm_consistency1}}\label{subsubsec:4-3}

Let $\psi_{\eta,j}(x_j|w_j)$ be the marginal density of the $j$th
coordinate, that is $\psi_{\eta,j}(x_j|w_j) = \int f(x_j - b^T w_j) dG(b)$.
Since $G_0$ is thick at 0 and $f_0$ is continuous and positive at 0,
there exists a $\gamma > 0$ such that $\inf_{|x_j| \leq \gamma} 
\inf_{w_j}\psi_{\eta_0,j}(x_j|w_j) > 0$. Thus, as in \eqref{eq:yd3},
there exist constants $\widetilde C > 0$ and $\delta > 0$ such that,
\[
  h^2 (p_{\theta,\eta,i}, p_{\theta_0,\eta_0,i})
  \geq \widetilde C^2\big(\delta \wedge 
  |(\theta-\theta_0)^T Z_{ij}|\big)^2,
\]
for every $j \leq m$. Since $\max_{j\leq m}|(\theta-\theta_0)^T
Z_{ij}| \geq |(\theta-\theta_0)^T Z_i| / \sqrt{m}$,
\[
  h^2 (p_{\theta,\eta,i}, p_{\theta_0,\eta_0,i})
  \geq C^2\big(\delta \wedge 
  |(\theta-\theta_0)^T Z_i|\big)^2,
\]
where $C = \widetilde C / \sqrt{m}$.
Let $\epsilon >0$ be a constant such that 
$\epsilon < aC\delta$, where $a^2 = \liminf_n\rho_{\rm min}(\bZ_n)/(2mL^2)$, 

For a given large enough $n$, fix $\eta\in\cH_n$ with
$h_n\big((\theta,\eta),(\theta_0,\eta_0)\big) < \epsilon$.
Let $\NN_{\delta,n} = \{i \leq n: |(\theta-\theta_0)^TZ_i| \geq \delta\}$
and let $N_{\delta,n}$ denote its cardinality. 
Then, the last display implies,
\bea
  \epsilon^2 \geq h_n^2((\theta,\eta),(\theta_0,\eta_0)) 
  &\geq& \frac{C^2}{n} \sum_{i=1}^n \big(|(\theta-\theta_0)^T Z_i|
    \wedge \delta\big)^2\nonumber\\
  &\geq& \frac{C^2 N_{\delta,n} \delta^2}{n} +
    \frac{C^2}{n}\sum_{i\notin\NN_{\delta,n}}|(\theta-\theta_0)^TZ_i|^2.
\label{eq:yd6}
\eea    
The first term of \eqref{eq:yd6} is greater than
$N_{\delta,n}\epsilon^2/(na^2)$ since $\epsilon< aC\delta$, which
implies $N_{\delta,n}/n< a^2$.
On the other hand, for the second term of \eqref{eq:yd6}, note that,
\[
\sum_{i\notin\NN_{\delta,n}}|(\theta-\theta_0)^TZ_i|^2
\geq \sum_{i=1}^n|(\theta-\theta_0)^TZ_i|^2
- N_{\delta,n} \max_i |(\theta-\theta_0)^TZ_i|^2.
\]
Since $\sum_{i=1}^n|(\theta-\theta_0)^TZ_i|^2 \geq n |\theta-\theta_0|^2
\rho_{\rm min}(\bZ_n)$ and $\max_i |(\theta-\theta_0)^TZ_i|^2
\leq m L^2 |\theta-\theta_0|^2$, we have,
\be
\label{eq:yd7}
\frac{C^2}{n}\sum_{i\notin\NN_{\delta,n}}|(\theta-\theta_0)^TZ_i|^2
\geq C^2 |\theta-\theta_0|^2 \Big( \rho_{\rm min}(\bZ_n)
    - m L^2 \frac{N_{\delta,n}}{n}\Big). 
\ee
Since $N_{\delta,n}/n< a^2$ and $a^2=\liminf_n
\rho_{\rm min}(\bZ_n)/(2mL^2)$, \eqref{eq:yd6} and \eqref{eq:yd7}
together imply $|\theta-\theta_0|^2 \leq K_1 \epsilon^2$, where
$K_1 = 2/\big(C^2 \rho_{\rm min}(\bZ_n)\big)$.

The proof would be complete if we show that $h_n(\eta, \eta_0)
< K \epsilon$ for some constant $K>0$. Note that for every $i$,
\be
\begin{split}\nonumber
  h_n^2(\eta,\eta_0) &= \frac{1}{n}
    \sum_{i=1}^n h^2(p_{\theta,\eta,i}, p_{\theta,\eta_0,i})\\
  & \leq \frac{2}{n} \sum_{i=1}^n \bigl(h^2(p_{\theta,\eta,i},
      p_{\theta_0,\eta_0,i})
    + h^2(p_{\theta_0,\eta_0,i}, p_{\theta,\eta_0,i})\bigr)\\
  &= \frac{2}{n} \sum_{i=1}^n h^2(p_{\theta,\eta_0,i}, p_{\theta_0,\eta_0,i})
    + 2 h_n^2((\theta,\eta), \theta_0,\eta_0))\\
  &\leq \frac{2}{n} \sum_{i=1}^n h^2(p_{\theta,\eta_0,i},
      p_{\theta_0,\eta_0,i})
    + 2\epsilon^2,
\end{split}
\ee
Note also that,
\[
\frac{\partial}{\partial\theta} \sqrt{p_{\theta,\eta,i}(x)}
  = \frac{\frac{\partial}{\partial \theta} p_{\theta,\eta,i}(x)}
    {2 \sqrt{p_{\theta,\eta,i}(x)}}
  = \frac{1}{2} Z_i s_\eta(x-Z_i^T\theta | W_i) \sqrt{p_{\theta,\eta,i}(x)}
\]
Thus, with $\theta(t) = \theta_0 + t(\theta-\theta_0)$, 
\be
\label{eq:h-theta-lm}
\begin{split}
  & h^2(p_{\theta,\eta_0,i}, p_{\theta_0,\eta_0,i}) 
    = \int \Big(\sqrt{p_{\theta,\eta_0,i}(x)} -
    \sqrt{p_{\theta_0,\eta_0,i}(x)} \Big)^2 d\mu(x)\\
  &\leq \frac{1}{4} \int \int_0^1 \Big| (\theta-\theta_0)^T 
    Z_i s_{\eta_0} \big(x-Z_i^T \theta(t)|W_i\big) \Big|^2
    p_{\theta(t),\eta_0,i}(x) \; dt d\mu(x)\\
  &= \frac{1}{4} \int \Big| (\theta-\theta_0)^T Z_i
    s_{\eta_0}(x|W_i) \Big|^2 d\Psi_{\eta_0}^{W_i}(x)\\
  &\leq K_2 |\theta-\theta_0|^2,
\end{split}
\ee
for some $K_2 > 0$ by \eqref{eq:Lipsc_cond_lm}, where the inequality
in the second line of \eqref{eq:h-theta-lm} holds by Jensen's inequality.
Thus, $h_n^2(\eta,\eta_0) \leq 2K_2 |\theta-\theta_0|^2
+ 2\epsilon^2 \leq 2(K_1 K_2+1)\epsilon^2$.
\qed

\subsubsection{Proof of Lemma \ref{lem:lm-consistency2}} \label{subsubsec:4-4}

Let $\eta_1=(f_1,G_1)$ and $\eta_2=(f_2,G_2)$ be elements of $\cH_0$,
$\theta_1,\theta_2\in\Theta$ and let $\eta_{12} = (f_1, G_2)$.
Since \eqref{eq:Q_0-integrability} and \eqref{eq:Q_0-lip} hold, it can
be shown, in a manner similar to \eqref{eq:h-theta-lm}, that,
\[
  \sup_{i\geq 1}\sup_{\eta\in\cH_0} h(p_{\theta_1, \eta, i}, p_{\theta_2, \eta,i})
    \leq C|\theta_1 - \theta_2|,
\]
for some constant $C > 0$.
Recall that $h(P*G, Q*G) \leq h(P, Q)$ for arbitrary probability
measures $P,Q$ and $G$ (where $*$ denotes convolution).
Then,
\[
\begin{split}
  h(p_{\theta_1, \eta_1, i}, &\; p_{\theta_2, \eta_2,i})
  \leq h(p_{\theta_1, \eta_1, i}, p_{\theta_1, \eta_2,i})
    + h(p_{\theta_1, \eta_2, i}, p_{\theta_2, \eta_2,i})\\
  &\leq h\big(\psi_{\eta_1}(\cdot|W_i), \psi_{\eta_2}(\cdot|W_i)\big)
    + C|\theta_1-\theta_2|\\
  &\leq h\big(\psi_{\eta_1}(\cdot|W_i), \psi_{\eta_{12}}(\cdot|W_i)\big)
    + h\big(\psi_{\eta_{12}}(\cdot|W_i), \psi_{\eta_2}(\cdot|W_i)\big)
    + C|\theta_1-\theta_2|\\
  &\leq h\big(\psi_{\eta_1}(\cdot|W_i), \psi_{\eta_{12}}(\cdot|W_i)\big)
    + h(f_1^m, f_2^m) + C|\theta_1-\theta_2|,\\
  &= h\big(\psi_{\eta_1}(\cdot|W_i), \psi_{\eta_{12}}(\cdot|W_i)\big) + o(1),
\end{split}\]
as $h(f_1,f_2) \vee |\theta_1-\theta_2| \rightarrow 0$, where $f^m$ is
the $m$-fold product density of $f$. To prove
\eqref{eq:h_n-consistency-lm}, it now suffices to show that
$h\big(\psi_{\eta_1}(\cdot|W_i), \psi_{\eta_{12}}(\cdot|W_i)\big) =
o(1)$ as $d_W(G_1, G_2) \rightarrow 0$.

By \eqref{eq:cF_0-bound}, there exists a constant $C_1 > 0$ such that,
\[
  \sup_{f\in\cF_0} \sup_{x,w} \bigg| \prod_{j=1}^m f(x_j-b_1^T w_j)
  - \prod_{j=1}^m f(x_j^T - b_2 w_j) \bigg| \leq C_1 |b_1 - b_2|,
\]
for every $b_1, b_2 \in [-M_b, M_b]^q$. So by Theorem 2 of
\cite{gibbs2002choosing} (the equivalence of L\'{e}vy-Prohorov
and Wasserstein metrics),
\be
\label{eq:d_W-bound}
  \sup_{x,w} |\psi_{\eta_1}(x|w) - \psi_{\eta_{12}}(x|w)|
  \leq C_2 d_W(G_1, G_2),
\ee
for some $C_2 >0$ that depends only on $C_1$.
Since $w$ ranges over a compact set, inequality \eqref{eq:d_W-bound}
and uniform tightness of $\cF_0$ imply that,
\[
  \sup_w d_V\big(\psi_{\eta_1}(\cdot|w), \psi_{\eta_{12}}(\cdot|w)\big)
    \rightarrow 0,
\]
as $d_W(G_1, G_2) \rightarrow 0$. Since $h^2 \leq d_V$, this
completes the proof of \eqref{eq:h_n-consistency-lm}.

To prove \eqref{eq:KL-consistency-lm}, write,
\be
  K(p_{\theta_0,\eta_0,i}, p_{\theta,\eta,i}) = K(\ell_{\theta_0,\eta_0,i},
  \ell_{\theta_0,\eta,i}) + P_0(\ell_{\theta_0,\eta,i} - \ell_{\theta,\eta,i}).
\ee
Under conditions \eqref{eq:Q_0-integrability} and 
\eqref{eq:Q_0-lip}, $K(p_{\theta_0,\eta_0,i}, p_{\theta_0,\eta,i})$ 
is bounded by,
\[
  C_3 h^2(p_{\theta_0,\eta_0,i}, p_{\theta_0,\eta,i}) 
  \log\bigg(\frac{1}{h(p_{\theta_0,\eta_0,i}, p_{\theta_0,\eta,i})}\bigg),
\]
for some constant $C_3 > 0$, by Theorem~5 of
\cite{wong1995probability}, which converges to 0 as $h(f, f_0) \vee
d_W(G, G_0) \rightarrow 0$ by \eqref{eq:h_n-consistency-lm}.
Also, by \eqref{eq:Q_0-lip},
\[
  \sup_{\eta\in\cH_0} P_0(\ell_{\theta_0,\eta,i} -
  \ell_{\theta,\eta,i})
  \leq C_4 |\theta-\theta_0|,
\]
for some $C_4 > 0$, and so $K(p_{\theta_0,\eta_0,i},
p_{\theta,\eta,i}) \rightarrow 0$ as $|\theta-\theta_0| \vee h(f, f_0)
\vee d_W(G, G_0) \rightarrow 0$. Similarly, 
\be
 V(p_{\theta_0,\eta_0,i}, p_{\theta,\eta,i}) \leq 
  2P_0(\ell_{\theta_0,\eta_0,i} - \ell_{\theta_0,\eta,i})^2
   + 2P_0(\ell_{\theta_0,\eta,i} -  \ell_{\theta,\eta,i})^2,
\ee
and $P_0(\ell_{\theta_0,\eta_0,i} - \ell_{\theta_0,\eta,i})^2$ is bounded by,
\[
  C_5 h^2(p_{\theta_0,\eta_0,i}, p_{\theta_0,\eta,i})
  \bigg\{
    \log\bigg(\frac{1}{h(p_{\theta_0,\eta_0,i},p_{\theta_0,\eta,i})}\bigg)
  \bigg\}^2,
\]
for some $C_5 >0$ by Theorem~5 of \cite{wong1995probability}.
In addition,
\be
  P_0(\ell_{\theta_0,\eta,i} - \ell_{\theta,\eta,i})^2
  \leq C_6 |\theta-\theta|^2,
\ee
for some $C_6 > 0$. Thus, $V(p_{\theta_0,\eta_0,i}, p_{\theta,\eta,i})
\rightarrow 0$ as $|\theta-\theta_0|$, $h(f, f_0)$ and $d_W(G, G_0)$
go to zero.
\qed

\subsubsection{Proof of  \eqref{eq:density_bound-lm}}\label{subsubsec:4-4.1}
For $x \in \RR^m$ and $w\in \RR^{q\times m}$ note that,
\bean
  \psi_\eta(x|w) &=& \int \prod_{j=1}^m \int 
    \phi_\sigma (x-z-b^T w_j ) dF(z,\sigma) dG(b)\\
  &=& \int \prod_{j=1}^m \int \frac{1}{\sqrt{2\pi}\sigma}
    \exp \bigg(-\frac{(x_j-z-b^T w_j)^2}{2\sigma^2}\bigg)dF(z,\sigma) dG(b)\\
  &\leq& (2\pi\sigma_1^2)^{-m/2} \exp\bigg(-\frac{|x|^2}{2\sigma_2^2}
    + K_1\bigg)\\
  &\leq& C_3 \exp(-C_4 |x|^2),
\eean
for $C_3 = (2\pi\sigma_1^2)^{-m/2}$, $C_4 < 1/(2\sigma_2^2)$ and
large enough $|x|$, where $K_1$ is a constant. In the same way,
\bean
  \psi_\eta(x|w) &\geq& (2\pi\sigma_2^2)^{-m/2}
   \exp\bigg(-\frac{|x|^2}{2\sigma_1^2} + K_2 \bigg)\\
  &\geq& C_1 \exp(-C_2 |x|^2),
\eean
for $C_1 = (2\pi\sigma_2^2)^{-m/2}$, $C_2 > 1/(2\sigma_1^2)$ and
large enough $|x|$, where $K_2$ is a constant.
\qed

\subsubsection{Proof of the equicontinuity of \eqref{eq:w_class}%
   in Corollary \ref{cor:ex-dpm-lm}}\label{subsubsec:4-5}

To prove the equicontinuity of \eqref{eq:w_class},
it is sufficient to show that the partial derivatives of
$w\mapsto d^2_w(\eta_1, \eta_2)$ and 
$w\mapsto h^2(\psi_{\eta_1}(\cdot|w), \psi_{\eta_2}(\cdot|w))$ 
are bounded by a constant uniformly in $\eta_1, \eta_2 \in \cH_0$.
Since every $G$ is compactly supported, partial derivatives of
$w \mapsto s_\eta(x|w)$ and $w\mapsto \psi_{\eta_0}(x|w)$ are
bounded by a constant multiple of partial derivatives of
$x \mapsto s_\eta(x|w)$ and $x\mapsto \psi_{\eta_0}(x|w)$, which are
bounded by $Q(x,w)$ and $Q(x,w) \psi_{\eta_0}(x|w)$, respectively.
Since $s_\eta(x,w)$ is also bounded by $Q(x,w)$ for every
$\eta \in\cH_0$, the partial derivative of,
\[
  w \mapsto d_w^2(\eta_1, \eta_2) = \int |s_{\eta_1}(x|w) - s_{\eta_2}(x|w)|^2
    \psi_{\eta_0}(x|w) d\mu(x),
\]
is bounded by a constant multiple of $\int Q^3(x,w) d\Psi_{\eta_0}^w(x)$.
Note that, 
\[
  h^2\big(\psi_{\eta_1}(\cdot|w), \psi_{\eta_2}(\cdot|w)\big) = 
  2\Big(1 - \int \sqrt{\psi_{\eta_1}(x|w) \psi_{\eta_2}(x|w)} d\mu(x)\Big).
\]
Since,
\[
  \frac{\partial\psi_\eta}{\partial w_j}(x|w) = 
  \int \Big(\dot f(x_j - w_j^T b)\prod_{k\neq j} f(x_k - w_k^T b)\Big)
  \cdot b \,dG(b),
\]
where $\dot f$ is the derivative of $f$, we have,
\begin{equation}
\label{eq:frac_eta}
  \left|\frac{\partial\psi_\eta(x|w)/\partial w_j}{\psi_\eta(x|w)}\right|
  \leq C \sup_b \left|\frac{\dot f(x_j - w_j^Tb)}{f(x_j - w_j^Tb)}\right|,
\end{equation}
for a constant $C > 0$, so the supremum of the left-hand
side of \eqref{eq:frac_eta} is of order $O(|x_j|)$,
as $|x_j| \rightarrow \infty$, where the supremum is taken over
$\eta\in\cH_0$.
Consequently,
\bean
  \left|\frac{\partial}{\partial w_j}
  h^2\big(\psi_{\eta_1}(\cdot|w), \psi_{\eta_2}(\cdot|w)\big)\right|
  &\leq& \left| \int \frac{\partial(\psi_{\eta_1}(x|w) 
  \psi_{\eta_2}(x|w))/\partial w_j}
  {\sqrt{\psi_{\eta_1}(x|w) \psi_{\eta_2}(x|w)}} d\mu(x) \right|\\
  &\leq& \int O(|x|) \times 
  (\psi_{\eta_1}(x|w) + \psi_{\eta_2}(x|w)) d\mu(x).
\eean
Since $\sup_w \sup_{\eta\in\cH_0} \int |x| d\Psi^w_{\eta}(x) < \infty$, this
establishes the equicontinuity of \eqref{eq:w_class}.
\qed

\subsubsection{Proof of asymtoptic tightness of%
   \eqref{eq:score_process_lm} in Corollary \ref{cor:ex-dpm-lm}}\label{subsubsec:4-6}

It only remains to prove asymptotic tightness of \eqref{eq:score_process_lm}.
Without loss of generality, we may assume that $\theta_0 = 0$.
Let $\scrF = B_\epsilon \times \cH_0$, where $B_\epsilon$ is the
Euclidean ball of radius $\epsilon$ centered on $\theta_0$, and
define $Z_{ni}(\theta,\eta)$, $S_{ni}$,  and $N_{[]}^n(\delta,\scrF)$
as in the proof of Corollary \ref{cor:ex-dpm}.
By the bracketing central limit theorem (Theorem~2.11.9 of
\cite{van1996weak}), it suffices to prove that,
\be
\begin{split}\label{eq:bracket_clt-lm}
  \sum_{i=1}^n P_0\big( S_{ni} 1_{\{S_{ni} > \gamma\}}\big) &= o(1),
    \quad\text{for every $\gamma > 0$},\\
  \int_0^{\delta_n}\sqrt{\log N^n_{[]} (\delta,\scrF)}\,d\delta &<
  \infty,
    \quad\text{for every $\delta_n \downarrow 0$}.
\end{split}
\ee
The first condition of \eqref{eq:bracket_clt-lm} is proved
in a manner similar to the proof of Corollary~\ref{cor:ex-dpm},
by replacing $Q_1$ by $Q$ defined in \eqref{eq:Q-dpm-lm}.

To prove the second condition of \eqref{eq:bracket_clt-lm},
note that,
\be
\label{eq:tri-lm-dpm}
\begin{split}
  |Z_{ni}(\theta_1, \eta_1) -& Z_{ni}(\theta_2, \eta_2)|\\
  &\leq |Z_{ni}(\theta_1,\eta_1) - Z_{ni}(\theta_2,\eta_1)|
    + |Z_{ni}(\theta_2,\eta_1) - Z_{ni}(\theta_2,\eta_2)|.
\end{split}
\ee
By \eqref{eq:Q-dpm-lm},
$\sup_{i \geq 1} \sup_{\eta\in\cH_0} |\score_{\theta,\eta,i}
- \score_{\theta_0,\eta,i}|$ is bounded by a constant multiple of
$|\theta_1 - \theta_2|$, and so there exists a constant $K_1 > 0$
such that,
\be
\label{eq:Z_ni-theta-lm}
  \sup_{\eta \in \cH_0} |Z_{ni}(\theta_1,\eta)
    - Z_{ni}(\theta_2,\eta)| \leq \frac{K_1}{\sqrt{n}}
    |\theta_1 - \theta_2|.
\ee
Let $g_{\eta,j} (x|w) = \partial \ell_\eta(x|w) / \partial x_j$, 
and for $t > 0$, let,
\[
  \cS_t = \{(x,w)\mapsto g_{\eta,j}(x|w) : \eta \in \cH_0, 1\leq j \leq m\},
\]
where functions in $\cS_t$ are viewed as maps from
$[-t,t]^{m} \times [-L,L]^{qm}$ to $\RR$.
Since $w$ ranges over a compact set and $G$ is supported on a
compact set, the $\alpha$-th order partial derivative of the map
$(x,w) \mapsto\psi_\eta(x|w)$ is bounded by a constant multiple of
$|x|^\alpha \psi_\eta(x|w)$ for every $\eta\in\cH_0$, $w$, and large
enough $|x|$. Thus, for some constant $D_\alpha >0$, the
$\alpha$-\Holder norm of functions in $\cS_t$ is bounded by
$D_\alpha t^{\alpha+1}$ for large enough $t$. Since the Lebesgue measure
of $[-t,t]^m\times[-L,L]^{qm}$ is bounded by a constant multiple
of $t^m$, applying Theorem~2.7.1 of \cite{van1996weak} with
$\alpha=d=(q+1)m$, there exists a constant $K_2 > 0$ such that,
\[
  \log N(\delta, \cS_t, \|\cdot\|_\infty) < K_2 \frac{t^{(q+2)m+1}}{\delta},
\]
for every $\delta > 0$ and large enough $t > 0$.
Since $\sup_w \sup_{\eta\in\cH_0} |s_\eta(x|w)| = O(|x|)$ and $\sup_w
\psi_{\eta_0}(x|w) = O(\exp(-K_3 |x|^2))$ as $|x|\rightarrow\infty$
for some constant $K_3 > 0$, we have,
\[
  \int_{\{|x| > M_\delta\}} \sup_{|y| \leq \sqrt{m}L\epsilon}
    \sup_w \sup_{\eta\in\cH_0}|s_\eta(x+y|w)|^2 
    \psi_{\eta_0}(x|w) d\mu(x) \leq \delta^2,
\]
for every small enough $\delta > 0$, where $M_\delta = -\log \delta$.
Therefore, for every small enough $\delta > 0$ we can construct a
partition $\cH_0 = \cup_{l=1}^{N_\delta} \cH^{(l)}$ such that, for
some constant $K_4 > 0$,
\[
  \log N_\delta \leq K_4 \bigg| \log \frac{1}{\delta}\bigg|^{(q+2)m+1}
  \frac{1}{\delta},
\]
and,
\be
\label{eq:Z_ni-eta-lm}
  \int \sup_{\theta\in B_\epsilon} \sup_{i \geq 1}
    \sup_{\eta_1, \eta_2 \in \cH^{(l)}} |Z_{ni}(\theta,\eta_1)
    - Z_{ni}(\theta,\eta_2)|^2 \psi_{\eta_0}(x|W_i) d\mu(x)
  < \frac{\delta^2}{n},
\ee
for every $l \leq N_\delta$. Since $N(\delta, B_\epsilon, |\cdot|)
= O(\delta^{-p})$ as $\delta \rightarrow 0$, \eqref{eq:tri-lm-dpm},
\eqref{eq:Z_ni-theta-lm} and \eqref{eq:Z_ni-eta-lm} imply that,
\[
  \log N^n_{[]}(\delta, \scrF) \leq K_5 \bigg(\bigg|
  \log \frac{1}{\delta}\bigg|^{(q+2)m+1} \frac{1}{\delta} 
    + \bigg| \log \frac{1}{\delta}\bigg|\bigg)
  \leq \bigg(\frac{1}{\delta}\bigg)^{3/2},
\]
for some $K_5 >0$, so the second condition of
\eqref{eq:bracket_clt-lm} is satisfied.
\qed

\subsubsection{Proof of asymptotic tightness of
  \eqref{eq:score_process_lm} in Corollary \ref{cor:ex-rs-lm}}
\label{subsubsec:4-7}

Without loss of generality, we may assume that $\theta_0 = 0$.
Let $\scrF = B_\epsilon \times \cH_0$, where $B_\epsilon$ is the
Euclidean ball of radius $\epsilon$ centered on $\theta_0$, and
define $Z_{ni}(\theta,\eta)$, $S_{ni}$,  and $N_{[]}^n(\delta,\scrF)$
as those defined in the proof of Corollary \ref{cor:ex-dpm-lm}.
By the bracketing central limit theorem (Theorem~2.11.9 of
\cite{van1996weak}), it is sufficient to prove \eqref{eq:bracket_clt-lm}.
The first condition of \eqref{eq:bracket_clt-lm} is easily satisfied.
For the second condition of \eqref{eq:bracket_clt-lm},
the inequalities \eqref{eq:tri-lm-dpm} and \eqref{eq:Z_ni-theta-lm}
hold similarly. Thus for every $\delta > 0$, it suffices to construct
a partition $\cH_0 = \cup_{l=1}^{N_\delta} \cH^{(l)}$ satisfying
\eqref{eq:Z_ni-eta-lm} and $\log N_\delta \leq \delta^{-2+\beta}$
for some $\beta > 0$.

For $f \in \cF_0$, let $\score_f$ be the derivatives of $\log f$, and let,
\[
  \cL = \cF_0 \cup \{\score_f : f\in \cF_0\}.
\]
Since functions in $\cL$ and their derivatives are uniformly bounded,
applying Theorem~2.7.1 of \cite{van1996weak} with $\alpha=d=1$,
there exists a constant $K_1 > 0$ such that,
\be
\label{eq:cL-entropy-rs-lm}
  \log N(\delta, \cL, \|\cdot\|_\infty) \leq K_1 \frac{1}{\delta},
\ee
for every $\delta > 0$.
Since $\cG_0$ is parametrized by a covariance matrix $\Sigma$, it
is compact in the induced matrix norm $\|\cdot\|$, and,
\[
  \log N(\delta, \cG_0, \|\cdot\|) \leq K_2 \log \bigg(\frac{1}{\delta}\bigg),
\]
for some $K_2 > 0$.
Note that in $\cG_0$, $d_V$ is bounded by a constant multiple of
$\|\cdot\|$ because the density of the normal distribution $N(0,\Sigma)$
is differentiable and its derivative is uniformly bounded, because
$\rho_{\min}(\Sigma) > \rho_1$. Since $d_W \leq d_V$ (see
\cite{gibbs2002choosing}), we have that,
\be
  \label{eq:cG0-entropy-rs-lm}
  \log N(\delta, \cG_0, d_W) \leq K_3 \log \bigg(\frac{1}{\delta}\bigg),
\ee
for some $K_3 > 0$.
Note that,
\[
  \frac{\partial\ell_\eta}{\partial x_k}(x|w) 
    = \frac{\int \score_f (x_k - b^T w_k)
        \prod_{j=1}^m f(x_j - b^T w_j) dG(b)}
      {\int \prod_{j=1}^m f(x_j - b^T w_j) dG(b)}.
\]
Since the denominator of the last display is bounded away from zero and
$b/a - d/c = b(c-a)/ac + (b-d)/c$ for every real numbers $a,b,c,d$
with $ac \neq 0$, there exists a constant $K_4 > 0$ such that for
every $f_1, f_2 \in \cF_0$,
\be
  \sup_{G\in\cG} \sup_{x,w} \bigg|
    \frac{\partial\ell_{(f_1,G)}}{\partial x_k}(x|w) 
    - \frac{\partial\ell_{(f_1,G)}}{\partial x_k}(x|w)\bigg|
  \leq K_4 (\|f_1 - f_2\|_\infty \vee \|\score_{f_1} - \score_{f_2}\|_\infty).
\ee
Also, by Theorem 2 of \cite{gibbs2002choosing}, there exists a
constant $K_5 > 0$ such that for every $G_1, G_2 \in \cG$,
\be
  \sup_{f\in\cF_0} \sup_{x,w} \bigg|
    \frac{\partial\ell_{(f,G_1)}}{\partial x_k}(x|w) 
      - \frac{\partial\ell_{(f,G_2)}}{\partial x_k}(x|w)\bigg|
  \leq K_5 d_W(G_1, G_2).
\ee
Therefore, by \eqref{eq:cL-entropy-rs-lm} and
\eqref{eq:cG0-entropy-rs-lm}, \eqref{eq:Z_ni-eta-lm} is satisfied
with entropy bound,
\[
  \log N_\delta \leq K_6 \bigg(\frac{1}{\delta}
    + \log\bigg(\frac{1}{\delta}\bigg)\bigg),
\]
for some $K_6 > 0$.
\qed
 
\section*{Acknowledgements}
The first author thanks to thesis committee members for valuable suggestions.
BK also thanks the {\it Statistics Department 
of Seoul National University, South Korea} for its kind hospitality.


\bibliographystyle{apalike}
\bibliography{bibliography}


\end{document}